\documentclass[12pt]{amsart}
\usepackage{amscd,amssymb,verbatim}
\theoremstyle{plain}

\usepackage{color}
\usepackage{marginnote}
\makeatletter
\g@addto@macro\@mn@margintest{%
  \@tempswatrue
  \@reversemarginfalse
  \setlength{\marginparwidth}{2.5cm}%
  \setlength{\marginparsep}{2mm}%
}
\makeatother

\numberwithin{equation}{section}

\newcommand{\N}{\mathbb{N}}
\newcommand{\R}{\mathbb{R}}

\newcommand{\s}{\sigma}

\DeclareMathOperator{\supp}{supp}
\DeclareMathOperator{\ran}{ran}

\newtheorem{theorem}{Theorem}[section]
\newtheorem{lemma}[theorem]{Lemma}
\newtheorem{remark}[theorem]{Remark}

\newtheorem{definition}[theorem]{Definition}
\newtheorem{proposition}[theorem]{Proposition}
\newtheorem{corollary}[theorem]{Corollary}
\newtheorem{notation}[theorem]{Notation}

\newtheorem*{question}{Question}

\begin{document}

\title[Shrinking Tree Bases]{Explicitly Defined Norms on $JT_*$ Spaces}

\author{Spiros A. Argyros}
\address{National Technical University of Athens, Faculty of Applied Sciences, Department of Mathematics, Zografou Campus, 157 80, Athens, Greece.}
\email{sargyros@math.ntua.gr}

\author{Pavlos Motakis}
\address{Department of Mathematics and Statistics, York University, 4700 Keele Street, Toronto, Ontario, M3J 1P3, Canada}
\email{pmotakis@yorku.ca}


\date{\today}

\begin{abstract}

We present a Banach space $X_\alpha$, equipped with an explicit norm, whose properties resemble those of the predual $JT_*$ of the James Tree space. In particular, $X_\alpha$ admits a shrinking dyadic basis with the property that the sums of its elements over any segment are uniformly bounded. The space $X_\alpha$ is complementably $\ell_2$-saturated, and its dual $X_\alpha^*$  does not contain an isomorphic copy of $\ell_1$. Finally, the bidual $X_\alpha^{**}$ is also complementably $\ell_2$-saturated, and the quotient $X_\alpha^{**}/X_\alpha$ is isometric to $c_0(2^{\mathbb{N}})$.
\end{abstract}
\maketitle

\setcounter{tocdepth}{1}
\tableofcontents

\section*{Introduction}
The James Tree space ($JT$), introduced by R. C. James \cite{J}, is a classical Banach space that has played a significant role in the evolution of Banach space theory. Even modern constructions, such as the Gowers Tree space \cite{G}, base their definitions on the structural properties of $JT$. This space,  its predual and dual ($JT_*, JT^{*}$) have been extensively studied by numerous authors. We mention the fundamental work of J. Lindenstrauss and C. Stegall \cite{LS}, the beautiful combinatorial result of I. Amemiya and I. Ito \cite{AI}, the proof of the primarity of $JT$ by A. D. Andrew \cite{andrew}, the comprehensive monograph by H. Fetter and B. Gamboa de Buen \cite{FetterG}, and our recent joint work with M. Gonz\'alez on the isomorphic structure of the subspaces of $JT$ and $JT^*$ \cite{AGM}.

The aim of the present paper is to provide an alternative definition for $JT$-type spaces. In particular, we define an explicitly described norm on $c_{00}(T)$, where $T$ is the dyadic tree, such that the corresponding completion $X_\alpha$ resembles $JT_*$, the predual of the classical James Tree space $JT$. It is a well-known fact that while $JT_*$ admits a shrinking basis, there is no explicit formula for its norm. This is primarily due to the absence of an explicit description for the norm of the dual of the classical James space $J$ \cite{J1} (whether $J$ is endowed with the squared variation norm or the boundedly complete norm \cite{J2, L1}). Our approach is motivated by the resolution of a problem concerning the existence of a Banach space in which every Schauder basic sequence is neither boundedly complete nor shrinking. The norm we present is rooted in a more advanced construction introduced in \cite{AM}, where a norming set is defined such that the induced space contains neither a boundedly complete basic sequence nor $c_0$. Furthermore, the present approach serves as an accessible introductory step for studying the more intricate results in \cite{AM}. 

It is worth pointing out that spaces endowed with explicitly or implicitly defined norms yielding a non-reflexive space with a shrinking basis are indispensable tools for studying certain problems. This is evident for the space presented in \cite{AM18}, which is an $\mathcal{L}_\infty$-space containing neither $\ell_1$, $c_0$, nor a reflexive subspace. Such a space necessarily lacks a boundedly complete basic sequence. In fact, the norm of the space in \cite{AM18} is based on the results and methods developed in \cite{AM}.

In what follows, we describe the proposed norm and explain the necessity of its ingredients. Our primary objective is to define a norm on $c_{00}(T)$ such that the standard tree basis is shrinking, the sums of the basis elements over any segment of the tree are uniformly bounded,  the resulting space is $\ell_2$-saturated and its dual does not contain $\ell_1$.

A natural first attempt to define a shrinking norm on $c_{00}(T)$ is to extend the variational norm of $J$ to the dyadic tree. Thus, one might define for $x \in c_{00}(T)$:
\begin{equation}
\|x\|_1 = \max\left\{ \| x\|_\infty, \quad \sup \left\{ \left( \sum_{i=1}^n |x(I_{i, \max}) - x(I_{i, \min})|^2 \right )^{1/2} \right\} \right\}
\end{equation}
where the supremum is taken over families $\{I_i\}_{i=1}^n$ of pairwise disjoint segments of $T$.

This norm fails to satisfy our second requirement, as it blows up the norm of the sums on segments of $T$. More precisely, as detailed in Section 1, for any finite segment $s$ of $T$, we have $\| \sum_{\alpha \in s} e_\alpha \|_1 \ge \sqrt{\#s}$. It is interesting that the space $X_1$  the completion of $(c_{00}, \|\cdot \|_1)$ is actually  isomorphic to the Hilbert space $\ell_2(T)$ via the natural correspondence of their bases. A variant of this norm, which imposes a topological condition to keep the sums on segments uniformly bounded, also fails because it forces the space to contain an isomorphic copy of $c_0$, violating the $\ell_2$-saturation requirement.

Let us now describe the norm defined in this paper. First, we identify the dyadic tree $T$ with $(\mathbb{N}, \preceq)$, where $\preceq$ is a partial order on $\mathbb{N}$ compatible with the usual order (i.e., $n \preceq m \implies n \le m$), structuring $(\mathbb{N}, \preceq)$ as a dyadic tree. We define the norm on $c_{00}(\mathbb{N})$ by evaluating vectors against specific dual functionals. To this end, we introduce two types of averages in $c_{00}(\mathbb{N})^*$:

\textbf{Comparable averages:}
\begin{equation}
\alpha_c^* = \frac{\sum_{l=1}^k (e_{n_l}^* - e_{m_l}^*)}{2k}
\end{equation}
where $n_1 \prec m_1 \prec n_2 \prec m_2 \prec \dots \prec n_k \prec m_k$.

\textbf{Incomparable averages:}
\begin{equation}
\alpha_{inc}^* = \frac{\sum_{l=1}^k e_{n_l}^*}{k}
\end{equation}
where $n_1, n_2, \dots, n_k$ are pairwise incomparable in the partial order of $\mathbb{N}$.

A sequence $\alpha_1^*, \dots, \alpha_k^*$ forms a \textbf{successive very fast growing (vfg) family of averages} if each $\alpha_i^*$ is either a comparable average $\alpha_c^*$ or an incomparable average $\alpha_{inc}^*$, they are strictly successive blocks in $\mathbb{N}$, and their sizes increase very fast.

The desired norm is then defined as follows. For $x \in c_{00}(\mathbb{N})$, we set:
\begin{equation}
\|x\|_\alpha = \max \left\{ \| x \|_\infty ,\, \sup \left\{ \left( \sum_{l=1}^k \|x\|_{\alpha_{I_l, l}^*}^2 \right)^{1/2} : \{\alpha_l^*\}_{l=1}^k \text{ vfg and } \{I_l\}_{l=1}^k \text{ intervals of } \mathbb{N} \right\} \right\}
\end{equation}
Here, for an average $\alpha^*$ and an interval $I$ of $\mathbb{N}$, $\| x \|_{\alpha_I^*} = |\alpha^*|_I (x)|$. The space $X_\alpha$ is the completion of $(c_{00}(\mathbb{N}), \| \cdot \|_\alpha)$.

Norms based on very fast growing families of averages appeared, in the saturated norm setting, in two papers of E.Odell and Th. Schlumprecht  \cite{OS95, OS00}. Subsequently were used in \cite{AM1} to provide the first (reflexive) space with the Hereditary Invariant Subspace Property and in \cite{ AM, AM18} in defining a $\mathcal{L}_\infty $ with the Scalar Plus Compact Property and not containing $c_{0}, \ell_1 $ or  a reflexive subspace. In the present paper we use these families in the definition of a classical norm.The use of averages is the key tool to avoid the branching leak that appeared in the first attempt to define a shrinking tree basis. 

 A crucial distinction between this norm and the classical James Tree norm is that our definition fundamentally relies on \emph{successive} very fast growing families of averages. In contrast, the James norm utilizes arbitrary families of pairwise disjoint segments, which is a key ingredient in establishing the properties of $JT$.

Next, we state some of the main properties of the space $X_\alpha$:
\begin{enumerate}
\item The sums of the elements of the basis $(e_n)_n$ on the finite segments of $(\mathbb{N}, \preceq)$ are uniformly bounded.
\item Every normalized block sequence in $X_\alpha$ admits an upper $\ell_2$ estimate with constant $4$.
\item Every normalized block sequence is weakly null.
\item The basis $(e_n)_n$ is shrinking.
\item Every closed infinite-dimensional subspace $Y$ of $X_\alpha$ contains a further subspace isomorphic to $\ell_2$ and complemented in $X_\alpha$.
\item The basis $(e_n)_n$ of $X_\alpha$ is a $c_0$ spreading model.
\item Every normalized block sequence in $X_\alpha$ either has a subsequence equivalent to the $\ell_2$ basis (spanning a space complemented in $X_\alpha$) or has a subsequence which is a $c_0$ spreading model.
\end{enumerate}
Properties $(1)$ to $(5)$ remain also valid in $JT_*$ while $(6)$ and $(7)$ are consequence of the use of averages in the definition of the norm and do not hold for $JT_*$.

Concerning the dual space $X_\alpha^*$, which serves as an analog to $JT$, we prove the following:
\begin{enumerate}
\item Every normalized block sequence in $X_\alpha^*$ admits a lower $\ell_2$ estimate.
\item The basis $(e_n^*)_n$ is an $\ell_1$ spreading model.
\item The space $X_\alpha^*$ does not contain an isomorphic copy of $\ell_1$.
\end{enumerate}

Our knowledge regarding the structure of $JT$ is much richer. In particular, results in \cite{AAK:08, AGM} and earlier works provide a complete description of the isomorphic structure of $JT$. Some of these properties include:
\begin{enumerate}
\item Every normalized block sequence in $JT$ admits a lower $\ell_2$ estimate.
\item Every normalized basic sequence in $JT$ has a subsequence equivalent either to the $\ell_2$ basis or to the basis of $J$ and the generated space is complemented in $JT$
\item Every subspace $Y$ of $JT$ with a non-separable dual contains an isomorphic copy of $JT$ which is complemented in $JT$.
\end{enumerate}

Unlike $X_\alpha^*$, the norm in $JT$ is explicitly defined, providing powerful tools for a better understanding of the space's structure. In contrast, proving that $\ell_1$ does not embed into $X_\alpha^*$ requires delicate, non-trivial arguments.

Finally, we study the bidual $X_\alpha^{**}$ and show the following:
\begin{enumerate}
\item The quotient space $X_\alpha^{**}/X_\alpha$ is isometric to $c_0(2^{\mathbb{N}})$. (In the case of $JT$, the corresponding quotient $JT^*/JT_*$ is isometric to $\ell_2(2^{\mathbb{N}})$).
\item The canonical quotient map $Q : X_\alpha^{**} \to X_\alpha^{**}/ X_\alpha$ is a strictly singular operator.
\item The space $X_\alpha^{**}$ is complementably $\ell_2$-saturated. 
\end{enumerate}
In Section 4 we present the line variant of the norm which defines a quasireflexive space $X_c$ with a shrinking basis and $ \text{dim} (X_c^{**} / X_c )= 1$. Namely the space $X_c$ is similar to the classical James space  $J$ (\cite{J1}).

As it is natural to expect proving, even analogue, results in the triples $(JT_*, JT, JT^*)$ and 
$(X_\alpha, X_\alpha^*, X_\alpha^{**})$ we follow different techniques, as the two triples have different starting space.

We use the standard Banach spaces notation as it appears in \cite{AK06, LT77}.

\begin{center}
\textbf{Declaration of Generative AI and AI-assisted technologies}
\end{center}

Throughout this research, the authors collaborated closely with the AI models Google Gemini and OpenAI ChatGPT. Both tools were used to format and typeset the \LaTeX\ manuscript under the authors' systematic supervision. Beyond typesetting, Google Gemini suggested the two norms discussed in Section 1 and formulated the statement and proof showing that the space $X_1$ is isomorphic to a Hilbert space (Proposition \ref{P5}). It also assisted in generating the statements and proofs in Section 4. After using these tools, the authors reviewed and edited the content as needed and take full responsibility for the final content of the publication.

\section{Discussion on Unsuccessful Attempts to Define the Norm}
In this section, we discuss two natural approaches to defining a norm on the dyadic tree and analyze the specific reasons they fail to yield the desired properties. The first attempt fails due to uncontrolled sums over segments, while the second inadvertently embeds $c_0$. Some of the standard formal terminology related to the dyadic tree used here is introduced later, in Section 2.

We wish to construct a norm $\|\cdot\|$ on the vector space $c_{00}(T)$ of finitely supported functions on the dyadic tree $T$ such that:
\begin{enumerate}
    \item[(1)] The standard unit vector basis $(e_{\alpha})_{\alpha \in T}$ is shrinking.
    \item[(2)] The sums of the basis elements over any segment are uniformly bounded; i.e., there exists $C > 0$ such that for every segment $s \subset T$, we have $\|\sum_{\alpha \in s} e_{\alpha}\| \le C$.
    \item[(3)] The completion is $\ell_2$-saturated.
    \item[(4)] The dual of the space does not contain $\ell_1$.
\end{enumerate}
Below, we present two natural but unsuccessful attempts to define such a norm, analyzing precisely why each fails.

\subsection{The First Attempt: Failure due to the Branching Property}
A natural starting point is to adapt the classical shrinking James norm $\|x\|_{J_s}$ by evaluating quadratic variation sums over disjoint segments of $T$. For any $x \in c_{00}(T)$, define:
\begin{equation}
\|x\|_1 = \max \left\{ \| x\|_\infty, \quad \sup \left\{ \left( \sum_{i=1}^n |x(I_{i, \max}) - x(I_{i, \min})|^2  \right)^{1/2} \right\} \right\} 
\end{equation}
where the supremum is taken over families $\{I_i\}_{i=1}^n$ of pairwise disjoint segments of $T$. 
\begin{proof}[Why this attempt fails]
The branching structure of the tree allows the variation term to ``leak'' norm, violating property (2). Let $s$ be an arbitrary segment of $T$. Consider the vector $x = \sum_{\alpha \in s} e_{\alpha}$. For every node $\alpha \in s$, let $I_{\alpha}$ be a segment of length greater than $1$ such that $I_{\alpha, \min} = \alpha$ and the rest of $I_{\alpha}$ branches off $s$ immediately (which is always possible since every node has two immediate successors). 

The family of segments $\{I_{\alpha}\}_{\alpha \in s}$ is pairwise disjoint. Evaluating the variation of $x$ on each $I_{\alpha}$, we obtain:
\begin{equation}
|x(I_{\alpha, \max}) - x(I_{\alpha, \min})| = |0 - 1| = 1
\end{equation}
By selecting these segments in our supremum, we get:
\begin{equation}
\|x\|_1 \ge \left( \sum_{\alpha \in s} |x(I_{\alpha, \max}) - x(I_{\alpha, \min})|^2 \right)^{1/2} = \sqrt{\# s}
\end{equation}
As the length of the segment $s$ increases, the norm $\|\sum_{\alpha \in s} e_{\alpha}\|_1$ grows arbitrarily large, meaning the sums of the basis elements over segments are not uniformly bounded.
\end{proof}
We set $X_1 $ be the completion of $(c_{00}, \| \cdot\|_1)$.
The following result was contributed by the AI when we asked it to prove that $X_1$ is reflexive, a fact we had previously verified. It demonstrates that the norm defined above is far from our intended goal.
\begin{proposition}\label{P5}
 The space  $X_1$  is isomorphic to the Hilbert space $\ell_2(T)$. 
\end{proposition}

\begin{proof}
We will show that there exist constants $c, C > 0$ such that 
$$c\|x\|_{\ell_2} \le \|x\|_1 \le C\|x\|_{\ell_2}$$ 
for all $x \in c_{00}(T)$.

\textbf{Upper bound:} Let $\{I_i\}_{i=1}^n$ be a family of pairwise disjoint segments in $T$, and let $\{\beta_j\}_{j=1}^m$ be a finite set of incomparable nodes disjoint from $\bigcup_{i=1}^n I_i$. Let $E = \bigcup_{i=1}^n \{ \min I_i, \max I_i \}$ be the set of endpoints of the segments. Since the segments are pairwise disjoint, the elements of $E$ are mutually distinct.

Using the elementary inequality $(a-b)^2 \le 2a^2 + 2b^2$, we obtain:
\begin{equation*}
\sum_{i=1}^n |x(\max I_i) - x(\min I_i)|^2 \le 2 \sum_{i=1}^n \left( |x(\max I_i)|^2 + |x(\min I_i)|^2 \right) = 2 \sum_{\alpha \in E} |x(\alpha)|^2.
\end{equation*}
Since the nodes $\{\beta_j\}_{j=1}^m$ are disjoint from $E$, we have:
\begin{equation*}
\sum_{i=1}^n |x(\max I_i) - x(\min I_i)|^2 + \sum_{j=1}^m |x(\beta_j)|^2 \le 2 \sum_{\alpha \in E} |x(\alpha)|^2 + \sum_{j=1}^m |x(\beta_j)|^2 \le 3 \|x\|_{\ell_2}^2.
\end{equation*}
Taking the supremum over all such families yields $\|x\|_1 \le \sqrt{3} \|x\|_{\ell_2}$.

\textbf{Lower bound:} We partition the edges of the tree into four disjoint families to construct valid test segments for the $\|\cdot\|_1$ norm. For any node $u \in T$, let $u_L$ and $u_R$ denote its immediate successors (left and right children). The four families of segments of length 1 are defined by:
\begin{enumerate}
    \item $F_1$: $\{ [u, u_L] : |u| \text{ is even} \}$
    \item $F_2$: $\{ [u, u_R] : |u| \text{ is even} \}$
    \item $F_3$: $\{ [u, u_L] : |u| \text{ is odd} \}$
    \item $F_4$: $\{ [u, u_R] : |u| \text{ is odd} \}$
\end{enumerate}
Within each family $F_k$, the segments are pairwise disjoint. Therefore, for each $k \in \{1, 2, 3, 4\}$, the sum of the squared variations over $F_k$ is bounded by $\|x\|_1^2$. Summing these four bounds gives:
\begin{equation*}
\sum_{u \in T} |x(u) - x(u_L)|^2 + \sum_{u \in T} |x(u) - x(u_R)|^2 \le 4 \|x\|_1^2.
\end{equation*}
We expand the left-hand side:
\begin{equation*}
V = \sum_{u \in T} \left( 2x(u)^2 + x(u_L)^2 + x(u_R)^2 - 2x(u)(x(u_L) + x(u_R)) \right).
\end{equation*}
Notice that $\sum_{u \in T} 2x(u)^2 = 2\|x\|_{\ell_2}^2$, and $\sum_{u \in T} (x(u_L)^2 + x(u_R)^2) = \|x\|_{\ell_2}^2 - x(\text{root})^2$. 
By the Cauchy-Schwarz inequality, the cross-term is bounded by:
\begin{equation*}
\left| \sum_{u \in T} 2x(u)(x(u_L) + x(u_R)) \right| \le 2 \left( \sum_{u \in T} x(u)^2 \right)^{1/2} \left( \sum_{u \in T} (x(u_L) + x(u_R))^2 \right)^{1/2}.
\end{equation*}
Using $(a+b)^2 \le 2(a^2+b^2)$, we have $\sum_{u \in T} (x(u_L) + x(u_R))^2 \le 2 \sum_{u \in T} (x(u_L)^2 + x(u_R)^2) \le 2\|x\|_{\ell_2}^2$.
Thus, the cross-term is bounded by $2 \|x\|_{\ell_2} \sqrt{2} \|x\|_{\ell_2} = 2\sqrt{2} \|x\|_{\ell_2}^2$.

Substituting this back into the expansion for $V$, we obtain:
\begin{equation*}
4 \|x\|_1^2 \ge V \ge 2\|x\|_{\ell_2}^2 + \|x\|_{\ell_2}^2 - x(\text{root})^2 - 2\sqrt{2}\|x\|_{\ell_2}^2 = (3 - 2\sqrt{2})\|x\|_{\ell_2}^2 - x(\text{root})^2.
\end{equation*}
Since the singleton set $\{\text{root}\}$ forms a valid choice for $\|x\|_\infty$ we also trivially have $\|x\|_1^2 \ge x(\text{root})^2$. Adding these two inequalities yields:
\begin{equation*}
5 \|x\|_1^2 \ge (3 - 2\sqrt{2})\|x\|_{\ell_2}^2.
\end{equation*}
Because $3 - 2\sqrt{2} > 0$, we deduce that $\|x\|_1 \ge c \|x\|_{\ell_2}$ for the constant $c = \sqrt{\frac{3 - 2\sqrt{2}}{5}}$.

Therefore, the spaces $(c_{00}(T), \|\cdot\|_1)$ and $(c_{00}(T), \|\cdot\|_{\ell_2})$ are linearly isomorphic. As a result, the completion $X_1$ is isomorphic to the Hilbert space $\ell_2(T)$.
\end{proof}
\subsection{The Second Attempt: Failure due to the $c_0$-embedding}
To resolve the branching leak of the first attempt, one can impose a topological constraint on where segments are allowed to start. Specifically, we can require that the starting elements of the variation segments must form an antichain (pairwise incomparable elements), preventing them from branching out simultaneously along the same segment.

Define the norm $\|\cdot\|_2$ on $c_{00}(T)$ as:
\begin{equation}
\|x\|_2 = \sup \left\{ \left( \sum_{j=1}^m \text{Var}(x|_{B_j})^2 \right)^{1/2} \right\}
\end{equation}
where $\{B_1, \dots, B_m\}$ are pairwise disjoint segments lying within individual branches, and:
\begin{equation}
\text{Var}(x|_{B_j}) = \sup_{p_1 \prec \dots \prec p_k \in B_j} \left( \sum_{r=1}^{k-1} |x(p_r) - x(p_{r+1})|^2 + |x(p_k)|^2 \right)^{1/2}
\end{equation}
with the crucial boundary constraint that the initial nodes $\{s_j = \min B_j\}_{j=1}^m$ form an \emph{antichain} in $T$.

\begin{proof}[Why this attempt fails]
While this boundary antichain constraint successfully bounds the segment sums (recovering property 2), the dyadic splitting of the tree under this formulation forces the completed space to contain an isomorphic copy of $c_0$, violating property (3) of $\ell_2$-saturation.

Let $D_n = \{\alpha \in T : |\alpha| = n\}$ be the set of all nodes of height $n$, which forms an antichain of size $2^n$. We define a normalized sequence of blocks $(z_n)_{n=1}^\infty$ by setting:
\begin{equation}
z_n = \frac{1}{2^{n/2}} \sum_{\alpha \in D_n} e_{\alpha}
\end{equation}
Choosing the branch segments to be the singletons $B_{\alpha} = \{\alpha\}$ for each $\alpha \in D_n$ (whose starting points are trivially an antichain), we obtain:
\begin{equation}
\|z_n\|_2^2 \ge \sum_{\alpha \in D_n} |z_n(\alpha)|^2 = 2^n \cdot \left(\frac{1}{2^{n/2}}\right)^2 = 1
\end{equation}
Thus, $\|z_n\|_2 \ge 1$ for all $n$.

Now let $x = \sum_{n=1}^N c_n z_n$. Let $\{B_1, \dots, B_m\}$ be any family of pairwise disjoint branch segments whose starting points $\{s_j\}_{j=1}^m$ form an antichain in $T$. Let $h_j = |s_j|$ denote the height of $s_j$. Since each $B_j$ lies on a single branch, it intersects at most one node of height $k \ge h_j$. For any node $p \in B_j$ of height $k = |p|$, the coordinate value of $x$ is $x(p) = c_k 2^{-k/2}$. Thus, the total variation along $B_j$ is bounded by:
\begin{equation}
\text{Var}(x|_{B_j}) \le\sqrt 2 \sum_{k \ge h_j} |c_k| 2^{-k/2} \le \sqrt 2 \left(\sup_{1 \le i \le N} |c_i|\right) \sum_{k \ge h_j} 2^{-k/2} \le C \left(\sup_{1 \le i \le N} |c_i|\right) 2^{-h_j/2}
\end{equation}
where $C = \frac{\sqrt 2}{1 - 1/\sqrt{2}}$. Squaring and summing these terms yields:
\begin{equation}
\sum_{j=1}^m \text{Var}(x|_{B_j})^2 \le C^2 \left(\sup_{1 \le i \le N} |c_i|^2\right) \sum_{j=1}^m 2^{-h_j}
\end{equation}
Since $\{s_j\}_{j=1}^m$ is an antichain in the dyadic tree, they satisfy the classical \emph{Kraft inequality} [\cite{CT06},
\cite{Kraft49}]
\begin{equation}
\sum_{j=1}^m 2^{-h_j} \le 1
\end{equation}
Applying this to the summation yields the upper bound:
\begin{equation}
\|x\|_2 \le C \sup_{1 \le i \le N} |c_i|
\end{equation}
Since $\|x\|_2 \ge \sup_{1 \le i \le N} |c_i| \cdot \|z_i\|_2 \ge \sup_{1 \le i \le N} |c_i|$, the sequence $(z_n)_{n=1}^\infty$ is equivalent to the unit vector basis of $c_0$, meaning the space is not $\ell_2$-saturated.
\end{proof}

\section{The Shrinking norm}
This section introduces the foundational concepts for our construction. We formally define the dyadic tree, introduce comparable and incomparable averages, and use very fast growing families of these averages to explicitly define the norming set and the space $X_\alpha$.

\subsection{The dyadic tree}
Recall that the dyadic tree $T$ is a countable tree with a unique root, and every $\alpha \in T$ has two immediate successors. Traditionally, $T$ is represented as $T=\{(n,i) : 0\leq n<\infty, 0\leq i <2^n \}$ with the usual partial order in which each $(n,i)\in T$ has two successors $(n+1,2i)$ and $(n+1,2i+1)$.
The \emph{height of} $\alpha=(n,i)\in T$ is $|\alpha|=n$.

Working in Banach spaces related to the dyadic tree, it is convenient to identify $T$ with $(\N, \preceq )$ where $\preceq$ is compatible with the natural order of $\N$ in the sense that $n \preceq m$ implies $n \leq m$. For example, the function $V : T \to \N$ defined by the rule $V((n,i)) = 2^n + i$ identifies $T$ with $(\N, \preceq)$.

A segment of $T$ is any set $s$ such that for every $n\leq m \in s$ we have $n \preceq m$, and for every $k\in \N$ with $n \preceq k \preceq m$ we have $k\in s$. We shall denote the segments by the letters $s, t$.
An initial segment is a segment having the root as its first element.
If $s$ is a segment and $I$ is an interval of $\N$, the restriction of $s$ to $I$ is the segment $s_{|I}=s\cap I$. This follows from the compatibility of $\preceq$ and $\leq$.

A branch is a maximal segment. We shall denote branches by the letters $\sigma,\tau$.
If $\sigma$ is a branch and $n\in\N$, we denote by $\sigma_{> n}$ the set $\sigma_{> n}=\{k\in \sigma: k>n\}$. For a branch $\sigma$ and $n\in\N$, we shall call the set $\sigma_{> n}$ a final segment.

Also, by $\sigma_{< n}$ we shall denote the initial segment $\sigma_{<n}=\{k\in \sigma: k<n\}$. In a similar manner, we define the initial segment $\sigma_{\leq n}$ and the final segment $\sigma_{\geq n}$.

Note that every infinite segment is contained in a unique branch. Actually, it is a final segment of a branch $\sigma$.
Two final segments $\sigma_{>n_{1}, 1}$ and $\sigma_{>n_{2}, 2}$ are said to be incomparable if their first elements are incomparable.
We denote by $\Gamma$ the set of all branches in $T$. Note that the cardinality of $\Gamma$ is $\mathfrak{c}$.

Let us note that:
\begin{enumerate}
    \item[a)] if $\sigma_{1},\dots,\sigma_{l}$ are distinct branches, then there exists $k\in\N$ such that the final segments $\sigma_{ > k, 1},\dots, \sigma_{>k, l}$ are incomparable.
    \item[b)] if $(\s_{n})_{n\in\N}$ is a sequence of initial segments, there exist a subsequence $(\s_{n})_{n\in L}$ and an initial segment or branch $\sigma$ such that $\s_{n}\xrightarrow[n\in L]{} \s$ in the pointwise topology.
\end{enumerate}

\subsection{Two Types of Averages}
To formulate a norm overcoming the limitations of previous attempts, we identify $T$ with $\mathbb{N}$ and define two distinct collections of functional averages in the dual space:

\begin{enumerate}
    \item[(i)] \textbf{Comparable Averages:} 
    Let $J = \{n_1, m_1, n_2, m_2, \dots, n_k, m_k\}$ be a finite sequence of comparable nodes in $T$ such that $n_1 \prec m_1 \prec n_2 \prec m_2 \prec \dots \prec n_k \prec m_k$. We associate to $J$ the functional:
    \begin{equation}
    \alpha_c^* = \frac{\sum_{l=1}^k (e_{n_l}^* - e_{m_l}^*)}{2k}
    \end{equation}
    We denote by $\mathcal{A}_c$ the collection of all such comparable averages.
    
    \item[(ii)] \textbf{Incomparable Averages:} 
    Let $I = \{n_1, \dots, n_k\}$ be a finite family of pairwise incomparable nodes in $T$. We associate to $I$ the functional:
    \begin{equation}
    \alpha_{inc}^* = \frac{\sum_{l=1}^k e_{n_l}^*}{k}
    \end{equation}
    We denote by $\mathcal{A}_{inc}$ the collection of all such incomparable averages.
\end{enumerate}

Let $\mathcal{A} = \mathcal{A}_c \cup \mathcal{A}_{inc}$ denote the union of these families. For any average $\alpha^* \in \mathcal{A}$ defined by a set $S \subset T$ (where $S$ is either $J$ or $I$), we set the \emph{size of the average} to be $s(\alpha^*) = \#S$. For $\alpha^* = \alpha_c^*$, $s(\alpha^*) = 2k$, and for $\alpha^* = \alpha_{inc}^*$, $s(\alpha^*) = k$.

\subsection{Very Fast Growing Families}

\begin{definition}[Very Fast Growing Families]
A finite family of successive averages $\{\alpha_1^*, \dots, \alpha_k^*\}$ is called \textbf{very fast growing (vfg)} if 
    for every $l > 1$, setting $h_{l-1} = \max \supp(\alpha_{l-1}^*)$, we have:
   
    \begin{equation}
    s(\alpha_l^*) > 2^{h_{l-1}}
    \end{equation}

\end{definition}

\begin{remark}\label{R2}
\begin{enumerate}
\item Notice that for every average $\alpha^* \in \mathcal{A}$ we always have $s(\alpha^*) \leq \max \supp(\alpha^*)$. Therefore, for every $\{\alpha_j^*\}_{j=1}^k$ vfg family, setting $\max \supp(\alpha_j^*) = h_j$, for $j > 1$ we have:
\[ h_j \geq s(\alpha_j^*) > 2^{h_{j-1}} \geq 2^{s(\alpha_{j-1}^*)} \]
In particular, $\{h_j\}_j$ and $\{s(\alpha_j^*)\}_j$ are strictly increasing and for $j > 1$:
\[ s(\alpha_j^*) > 2^{s(\alpha_{j-1}^*)} \]
\item If $\{\alpha_i^*\}_{i \in F}$ is a vfg family and $H \subset F$, then $\{\alpha_i^*\}_{i \in H}$ is also a vfg family.
\end{enumerate}
\end{remark}

\subsection{Definition of the Norm}
We define the norming set $G \subset c_{00}(T)^*$ to be the smallest subset of $c_{00}(T)^*$ satisfying the following properties:
\begin{enumerate}
    \item[(a)] $G$ contains all singletons $\{\pm e_n^*\}_{n \in T}$ and it is symmetric (i.e. if $x^* \in G$, then $-x^* \in G$).
    \item[(b)] $G$ is closed under projections on intervals of $\mathbb{N}$. Specifically, if $x^* \in G$ and $I$ is an interval of $\mathbb{N}$, then $P_I^*(x^*) = x^*|_I \in G$. 
    \item[(c)] For any very fast growing family of averages $\{\beta_i^*\}_{i \in F} \subset \mathcal{A}$, every family of intervals $\{I_i\}_{i \in F}$ of $\mathbb{N}$, and every sequence of scalars $\{\lambda_i\}_{i \in F}$ with $\sum_{i \in F} \lambda_i^2 \le 1$, setting $\alpha_i^* = \beta_i^*|_{I_i}$, the functional
    \begin{equation}
    x^* = \sum_{i \in F} \lambda_i \alpha_i^*
    \end{equation}
    belongs to $G$.
\end{enumerate}

\begin{definition}\label{D1}
For $x^* = \sum_{i \in F} \lambda_i \alpha_i^* \in G$ with $\sum_{i \in F} \lambda_i^2 \le 1$ and $\alpha_i^* = \beta_i^*|_{I_i}$, we call the vfg family $\{\beta_i^*\}_{i \in F}$ the \textbf{generators} of $x^*$.
\end{definition}

The norm on the space $c_{00}(T)$ is defined by:
\begin{equation}
\|x\|_\alpha = \sup \{ x^*(x) : x^* \in G \}
\end{equation}

An alternative form of the definition is the following:
\begin{equation}
\|x\|_\alpha = \max \left\{ \| x \|_\infty ,\, \sup \left\{ \left( \sum_{l=1}^k \|x\|_{\alpha_{I_l, l}^*}^2 \right)^{1/2} : \{\alpha_l^*\}_{l=1}^k \text{ vfg and } \{I_l\}_{l=1}^k \text{ intervals of } \mathbb{N} \right\} \right\}
\end{equation}
Here for $\alpha^* \in \mathcal{A}$ and $I$ an interval of $\mathbb{N}$, $\| x \|_{\alpha_I^*} = |\alpha^*|_I (x)|$.

It is easy to check that the two definitions determine the same norm.

The space $X_\alpha$ is the completion of $(c_{00}(\mathbb{N}), \| \cdot \|_\alpha)$.

\begin{notation}
For a branch $\sigma \in \Gamma$ and $x \in X_\alpha$, we set
\begin{equation}
P_\sigma(x) = x|_\sigma = \sum_{k \in \sigma} x(k)e_k
\end{equation}
\end{notation}

\begin{lemma}\label{lem_branch_projection}
For every branch $\sigma \in \Gamma$, the projection $P_\sigma : X_\alpha \to \overline{\operatorname{span}}\{e_k : k \in \sigma\}$ is a bounded operator with $\|P_\sigma\| = 1$.
\end{lemma}

\begin{proof}
Let $x \in X_\alpha$, and let $x^* \in G$ such that $x^*(P_\sigma(x)) > \|P_\sigma(x)\| - \epsilon$. Then $x^* = \sum_{i \in F} \lambda_i \alpha_i^*$ with $\sum_{i \in F} \lambda_i^2 \le 1$, and we may assume the following:

For every $i \in F$, $\supp(\alpha_i^*) \cap \sigma \neq \emptyset$ and $\supp(\alpha_i^*) \subset \sigma$.

The first condition is obvious, while for the second, if the generator $\beta_i^*$ of $\alpha_i^*$ is a comparable average, then it meets $\sigma$ up to a point and we may stop it at that point. If it is an incomparable average, then the intersection with $\sigma$ is a singleton and we may keep only this point.

Therefore, $x^*(x) = x^*(P_\sigma(x)) > \|P_\sigma(x)\| - \epsilon$, which yields that $\|x\| \ge \|P_\sigma(x)\|$, since $\epsilon > 0$ was arbitrary. Hence, $\|P_\sigma\| \le 1$. Since $P_\sigma(e_k) = e_k$ for $k \in \sigma$, it follows that $\|P_\sigma\| = 1$.
\end{proof}

\begin{lemma}
Let $\{\sigma_j\}_{j=1}^m$ be a finite family of distinct branches. Then for every $\epsilon > 0$, there exists $n \in \mathbb{N}$ such that the projection $P_{\cup_{j=1}^m \sigma_{>n,j}}$ is a bounded operator with $\|P_{\cup_{j=1}^m \sigma_{>n,j}}\| \le 1+\epsilon$.
\end{lemma}

\begin{proof}
Let $\epsilon > 0$. We select $n \in \mathbb{N}$ large enough such that $\frac{2m}{2^n} < \epsilon$ and $\{\sigma_{>n, j}\}_j$ are pairwise incomparable. We denote the union of the final segments by $S_n = \cup_{j=1}^m \sigma_{>n,j}$. Let $x \in X_\alpha$ and define $y = P_{S_n}(x)$. We may assume that $y\neq 0$.

We fix a functional $x^* = \sum_{i=1}^l \lambda_i \alpha_i^* \in G$ defined by a successive vfg family of generators $\{\beta_i^*\}_{i=1}^l$, with $\alpha_i^*=\beta_i^*|_{I_i}$ and $\sum_{i=1}^l\lambda_i^2\leq 1$. Because subsequences of vfg families are vfg, we may assume, without loss of generality, that $\alpha_i^*(y)\neq 0$ for $1\leq i\leq l$. We partition $\{1,\dots,l\}$ into
\[
F_1=\{1\}\cup\{i:\beta_i^*\in\mathcal{A}_c\},
\qquad F_2=\{i\geq 2:\beta_i^*\in\mathcal{A}_{inc}\},
\]
and set $x_r^*=\sum_{i\in F_r}\lambda_i\alpha_i^*$ for $r=1,2$. Both functionals belong to $G$.

For $i\in F_2$, an antichain meets each branch in at most one node, so $\supp(\beta_i^*)\cap S_n$ contains at most $m$ elements. Since $\alpha_1^*(y)\neq 0$, we have $\max\supp(\beta_1^*)>n$. Successiveness and the vfg condition therefore give, for $i\geq 2$,
\[
s(\beta_i^*)>2^{\max\supp(\beta_{i-1}^*)}>2^{n+i-2}.
\]
Consequently,
\[
|x_2^*(y)|\leq\sum_{i\in F_2}|\lambda_i\alpha_i^*(y)|
\leq\sum_{k=0}^\infty\frac{m}{2^{n+k}}\|y\|_\infty
=\frac{2m}{2^n}\|y\|_\infty
\leq\epsilon\|x\|_\alpha.
\]

To estimate $x_1^*(y)$, let $F_c=\{i\in F_1:\beta_i^*\in\mathcal{A}_c\}$. For $i\in F_c$, the chain supporting $\beta_i^*$ meets at most one of the incomparable final segments forming $S_n$, and its intersection with that final segment is an interval of the chain. As in the proof of Lemma 2.5, we can thus restrict $I_i$ further to an interval $J_i$ such that, setting $\widehat\alpha_i^*=\beta_i^*|_{J_i}$, we have
\[
\widehat\alpha_i^*(x)=\alpha_i^*(y).
\]
If $1\in F_c$, the functional $z^*=\sum_{i\in F_c}\lambda_i\widehat\alpha_i^*$ belongs to $G$ and satisfies $z^*(x)=x_1^*(y)$.

If $1\notin F_c$, set
\[
E=\supp(\alpha_1^*)\cap S_n,\qquad q=\#E,\qquad s=s(\beta_1^*),
\qquad \widetilde\beta_1^*=\frac1q\sum_{k\in E}e_k^*.
\]
Here $q>0$, since $\alpha_1^*(y)\neq 0$, and $\widetilde\beta_1^*$ is an incomparable average. With $\lambda_1'=\lambda_1q/s$, we have
\[
\lambda_1'\widetilde\beta_1^*(x)=\lambda_1\alpha_1^*(y).
\]
Replacing $\beta_1^*$ by $\widetilde\beta_1^*$ preserves successiveness and the vfg condition, because its support is a subset of $\supp(\beta_1^*)$. Also, $|\lambda_1'|\leq|\lambda_1|$. Hence
\[
z^*=\lambda_1'\widetilde\beta_1^*+\sum_{i\in F_c}\lambda_i\widehat\alpha_i^*\in G,
\qquad z^*(x)=x_1^*(y).
\]
In either case, $x_1^*(y)\leq\|x\|_\alpha$. Combining the two estimates yields
\[
x^*(y)=x_1^*(y)+x_2^*(y)
\leq(1+\epsilon)\|x\|_\alpha.
\]
Since $x^*\in G$ is arbitrary, we conclude that $\|P_{S_n}\|\leq 1+\epsilon$.
\end{proof}

\
\section{Properties of $X_\alpha$}
This section explores the primary structural properties of the space $X_\alpha$. We establish uniform bounds for segment sums, prove upper and lower $\ell_2$ estimates, and demonstrate that $X_\alpha$ admits a shrinking basis, contains  $c_0$ spreading models , and it is $\ell_2$-saturated.

\subsection{Uniform Boundedness of Segment Sums}
Using the properties of the very fast growing families, we can establish that the sums of the basis elements on any segment are uniformly bounded.

\begin{proposition}\label{prop_segment_sums}
Let $s \subset T$ be a nonempty finite segment of the tree. Then for the vector $x = \sum_{n \in s} e_n$, we have:
\begin{equation}
1\leq\|x\|\leq 2
\end{equation}
\end{proposition}

\begin{proof}
Let $x^* = \sum_{i=1}^d \lambda_i \alpha_i^* \in G$ be defined by a very fast growing family of generators $\{\beta_i^*\}_{i=1}^d$ with $\alpha_i^* = \beta_i^*|_{I_i}$ and $\sum \lambda_i^2 \le 1$. If every $\supp(\alpha_i^*)$ is disjoint from $s$, then $x^*(x)=0$. Otherwise, let $i_0$ be the first index such that $\supp(\alpha_{i_0}^*) \cap s \neq \emptyset$.

For a comparable average $\beta_i^*$ of size $2k$, the coordinates retained in $s\cap I_i$ form a consecutive portion of its alternating sum. Thus $|\alpha_i^*(x)|\leq 1/(2k)=1/s(\beta_i^*)$. For an incomparable average, its support meets $s$ in at most one node, giving the same bound. In particular,
\begin{equation}
|\alpha_{i_0}^*(x)| \le \frac{1}{s(\beta_{i_0}^*)} \le 1.
\end{equation}
For any subsequent term $i > i_0$, setting $h_{i-1}=\max\supp(\beta_{i-1}^*)$, the very fast growing property gives
\begin{equation}
|\alpha_i^*(x)| \le \frac{1}{s(\beta_i^*)} < \frac{1}{2^{h_{i-1}}}.
\end{equation}
Since the generators are successive, $h_{i-1}\geq i-i_0$. Therefore the sum over $i>i_0$ is bounded by $\sum_{r=1}^{\infty}2^{-r}=1$, and hence
\begin{equation}
\sum_{i=1}^d |\alpha_i^*(x)| \le 1+1=2.
\end{equation}
Applying the Cauchy-Schwarz inequality, we obtain:
\begin{equation}
x^*(x) \le \left( \sum_{i=1}^d \lambda_i^2 \right)^{1/2} \left( \sum_{i=1}^d |\alpha_i^*(x)|^2 \right)^{1/2} \le 2.
\end{equation}
Taking the supremum over all $x^* \in G$ yields $\|x\|\leq 2$. For any $n\in s$, we have $e_n^*(x)=1$, so $\|x\|\geq 1$.
\end{proof}

\subsection{\textbf{Upper and Lower $\ell_2$ Estimates}}
\begin{lemma} \label{L1} Let $x^* = \sum_{i \in F} \lambda_i \alpha_i^* \in G$ and $x \in X_\alpha$. Then:
\begin{equation}
x^*(x) \le \left( \sum_{i \in F} \lambda_i^2 \right)^{1/2} \|x\|_\alpha
\end{equation}
\end{lemma}

\begin{proof}
We simply observe that
\begin{equation}
y^* = \left( \sum_{i \in F} \lambda_i^2 \right)^{-\frac{1}{2}}x^* \in G \quad \text{hence} \quad y^*(x) \le \|x\|_\alpha
\end{equation}
\end{proof}

\begin{proposition} Let $(x_n)_n$ be a normalized block sequence in $X_\alpha$. Then it admits an upper $\ell_2 $ estimate. In particular
\begin{equation}
\left\| \sum_{n=1}^\infty b_nx_n \right\|_\alpha \le 4 \left(\sum_{n=1}^\infty b_n^2 \right)^{1/2 }
\end{equation}
\end{proposition}

\begin{proof}
Let $x = \sum_{n=1}^m b_n x_n$. We will show that for any functional $x^* = \sum_{i=1}^l \lambda_i \alpha_i^* \in G$ defined by successive generators $\{\beta_i^*\}_{i=1}^l$ and bounded coefficients $\sum_{i=1}^l \lambda_i^2 \le 1$, we have $|x^*(x)| \le 4 \left( \sum_{n=1}^m b_n^2 \right)^{1/2}$.

Let $E_n = \text{ran}(x_n)$ and $I_i = \text{ran}(\alpha_i^*)$. Because $(x_n)_{n=1}^m$ is a block sequence, the intervals $E_n$ are pairwise disjoint and successive. Because the generators $\{\beta_i^*\}_{i=1}^l$ are successive, the ranges $I_i$ are also pairwise disjoint and successive intervals in $\mathbb{N}$.

We partition the pairs of indices $(n, i)$ for which $\text{supp}(x_n) \cap \text{supp}(\alpha_i^*) \neq \emptyset$ into three disjoint subsets:
\begin{enumerate}
    \item $S_1 = \{ (n, i) : E_n \subset I_i \}$
    \item $S_2 = \{ (n, i) : I_i \subset E_n\}\setminus S_1$
    \item $S_3 = \{ (n, i) : E_n \cap I_i \neq \emptyset \text{ and neither contains the other} \}$
\end{enumerate}

\textbf{Step 1: Evaluation on $S_1$.} 
For each $i \le l$, let $F_i = \{n : (n, i) \in S_1\}$. Since the intervals $I_i$ are pairwise disjoint, each $n$ belongs to at most one $F_i$. 
For any $n \in F_i$, since $E_n \subset I_i$, we have $\alpha_i^*(x_n) = \beta_i^*(x_n)$. Letting $S(\beta_i^*)$ be the set of nodes defining the average $\beta_i^*$, we set $m_n = \#(S(\beta_i^*) \cap E_n)$. The definition of the averages gives:
\[ |\alpha_i^*(x_n)| \le \sum_{p \in S(\beta_i^*) \cap E_n} |\beta_i^*(e_p)| \cdot |x_n(p)| \le \frac{m_n}{s(\beta_i^*)} \|x_n\|_\infty \le \frac{m_n}{s(\beta_i^*)} \]
Because the blocks $E_n$ are disjoint, $\sum_{n \in F_i} m_n \le s(\beta_i^*)$, which implies $\sum_{n \in F_i} \frac{m_n}{s(\beta_i^*)} \le 1$.
Therefore:
\[ \left| \alpha_i^* \left( \sum_{n \in F_i} b_n x_n \right) \right| \le \sum_{n \in F_i} |b_n| \frac{m_n}{s(\beta_i^*)} \le \max_{n \in F_i} |b_n| \equiv |b_{n_i}| \]
Summing over all $i$:
\[ \left| \sum_{i=1}^l \lambda_i \alpha_i^* \left( \sum_{n \in F_i} b_n x_n \right) \right| \le \sum_{i=1}^l |\lambda_i| |b_{n_i}| \le \left( \sum_{i=1}^l \lambda_i^2 \right)^{1/2} \left( \sum_{i=1}^l b_{n_i}^2 \right)^{1/2} \le 1 \cdot \left( \sum_{n=1}^m b_n^2 \right)^{1/2} \]

\textbf{Step 2: Evaluation on $S_2$.} 
For each $n \le m$, let $H_n = \{i : (n, i) \in S_2\}$. Since the blocks $E_n$ are pairwise disjoint, the sets $H_n$ are pairwise disjoint.
For a fixed $n$, the functional $x_n^* = \sum_{i \in H_n} \lambda_i \alpha_i^*$ represents a restriction of $x^*$ entirely onto $E_n$, so by Lemma~\ref{L1}:
\[ |x_n^*(x_n)| \le \left( \sum_{i \in H_n} \lambda_i^2 \right)^{1/2} \|x_n\|_\alpha = \left( \sum_{i \in H_n} \lambda_i^2 \right)^{1/2} \]
Summing over all $n$:
\[ \left| \sum_{n=1}^m b_n x_n^*(x_n) \right| \le \sum_{n=1}^m |b_n| \left( \sum_{i \in H_n} \lambda_i^2 \right)^{1/2} \le \left( \sum_{n=1}^m b_n^2 \right)^{1/2} \left( \sum_{n=1}^m \sum_{i \in H_n} \lambda_i^2 \right)^{1/2} \le 1 \cdot \left( \sum_{n=1}^m b_n^2 \right)^{1/2} \]

\textbf{Step 3: Evaluation on $S_3$.} 
For pairs $(n, i) \in S_3$, the intervals $E_n$ and $I_i$ overlap by crossing boundaries. Because $\{E_n\}$ and $\{I_i\}$ are sequences of pairwise disjoint intervals, each $I_i$ can cross at most two $E_n$ blocks (one containing its left endpoint, one containing its right). Thus, for each $i$, there are at most two indices $n_1(i)$ and $n_2(i)$ such that $(n, i) \in S_3$.
Similarly, each $E_n$ can cross at most two $I_i$ intervals. Thus, for each $n$, there are at most two indices $i$ such that $(n, i) \in S_3$.
We bound this contribution:
\[ \sum_{(n, i) \in S_3} |\lambda_i b_n \alpha_i^*(x_n)| \le \sum_{i=1}^l |\lambda_i| (|b_{n_1(i)}| + |b_{n_2(i)}|) \]
Applying the Cauchy-Schwarz inequality:
\[ \le \left( \sum_{i=1}^l \lambda_i^2 \right)^{1/2} \left( \sum_{i=1}^l (|b_{n_1(i)}| + |b_{n_2(i)}|)^2 \right)^{1/2} \le 1 \cdot \left( 2 \sum_{i=1}^l \left( b_{n_1(i)}^2 + b_{n_2(i)}^2 \right) \right)^{1/2} \]
Since each $n$ appears at most twice in the collection of crossing pairs $S_3$, the term $b_n^2$ is summed at most twice. Thus:
\[ \left( 2 \sum_{i=1}^l \left( b_{n_1(i)}^2 + b_{n_2(i)}^2 \right) \right)^{1/2} \le \left( 4 \sum_{n=1}^m b_n^2 \right)^{1/2} = 2 \left( \sum_{n=1}^m b_n^2 \right)^{1/2} \]

\textbf{Conclusion.} 
Summing the evaluations over the three disjoint partitions yields:
\[ |x^*(x)| \le 1 \left(\sum_{n=1}^m b_n^2\right)^{1/2} + 1 \left(\sum_{n=1}^m b_n^2\right)^{1/2} + 2 \left(\sum_{n=1}^m b_n^2\right)^{1/2} = 4 \left(\sum_{n=1}^m b_n^2\right)^{1/2} \]
Taking the supremum over all $x^* \in G$ gives $\left\| \sum_{n=1}^\infty b_n x_n \right\|_\alpha \le 4 \left( \sum_{n=1}^\infty b_n^2 \right)^{1/2}$.
\end{proof}

\begin{corollary}\label{cor_weak_null_Xalpha}
Every normalized block sequence $(x_n)_n$ in $X_\alpha$ is weakly null. In particular, the space $X_\alpha$ does not contain an isomorphic copy of $\ell_1$.
\end{corollary}

\begin{proof}
The first part follows from every normalized block sequence admitting an upper $\ell_2$ estimate. Furthermore, if $\ell_1$ is embedded into $X_\alpha$, then by standard perturbation arguments there would exist a normalized block sequence equivalent to the $\ell_1$ basis, which is not weakly null.
\end{proof}

\begin{corollary}\label{C2}
The basis $(e_n)_n$ is a shrinking basis.
\end{corollary}

\begin{proof}
A basis of a Banach space is shrinking if and only if every bounded block sequence is weakly null. Since the previous corollary establishes that every normalized (and hence every bounded) block sequence is weakly null, it follows that the basis $(e_n)_n$ is shrinking.
\end{proof}

\begin{lemma}\label{L3}
Let $(x_n)_n$ be a normalized block sequence in $X_\alpha$ with $\|x_n\|_\infty \to 0$. Then for every sequence $(\epsilon_k)_k$ of positive reals, there exist a subsequence $(x_{n_k})_k$ and a sequence $(x_k^*)_k \subset G$ satisfying the following:
\begin{enumerate}
    \item For every $k \in \mathbb{N}$, $x_k^*(x_{n_k}) > 1 - \epsilon_k$ and $\text{supp}(x_k^*) \subset \text{ran}(x_{n_k})$.
    \item\label{L3:item2} For every $k \in \mathbb{N}$, $x_k^* = \sum_{j \in F_k} \lambda_{j,k} \alpha_{j,k}^*$ with generators $\{\beta_{j,k}^*\}_{j \in F_k}$ and $\sum_{j \in F_k} \lambda_{j,k}^2 \le 1$.
    \item\label{L3:item3} For every $m \in \mathbb{N}$, the family $\bigcup_{k \le m} \{\beta_{j,k}^* : j \in F_k\}$ is a successive very fast growing family of averages.
\end{enumerate}
\end{lemma}

\begin{proof}
We proceed by induction on $k$.

For $k = 1$, choose $n_1$ such that $\|x_{n_1}\|_\infty < \epsilon_1$, and let $y_1^* \in G$ with $y_1^*(x_{n_1}) = 1$ and $\text{supp}(y_1^*) \subset \text{ran}(x_{n_1})$. Write $y_1^* = \sum_{j \in F_1} \lambda_{j,1} \alpha_{j,1}^*$ with generators $\{\beta_{j,1}^*\}_{j \in F_1}$. We simply set $x_1^* = y_1^*$, which naturally satisfies $x_1^*(x_{n_1}) = 1 > 1 - \epsilon_1$ and its generators form a valid vfg family.

Assume that for $m > 1$ we have chosen $\{x_{n_1}, \dots, x_{n_{m-1}}\}$ and $\{x_1^*, \dots, x_{m-1}^*\}$ satisfying conditions (1), (2), and (3). For each $1 \le k \le m-1$, let $H_k = \{\beta_{j,k}^* : j \in F_k\}$ denote the generators of $x_k^*$. Set:
\[ h_{m-1} = \max \left\{ \max \supp(\beta_{j,k}^*) : 1 \le k \le m-1, \, j \in F_k \right\} \]

Since $\|x_n\|_\infty \to 0$, we choose $n_m > n_{m-1}$ such that $\min \text{ran}(x_{n_m}) > h_{m-1}$, and
\[ \|x_{n_m}\|_\infty < \epsilon_m. \]
We select $y_m^* \in G$ such that $y_m^*(x_{n_m}) = 1$ with $\text{supp}(y_m^*) \subset \text{ran}(x_{n_m})$. Write $y_m^* = \sum_{j =1 }^{N_m} \lambda_{j,m} \alpha_{j,m}^*$ with generators $\{\beta_{j,m}^* : 1\leq j \leq N_m\}$. Without loss of generality, we may assume that, for all $1\leq j\leq N_m$, $\alpha_{j,m}^*(x_{n_m})\neq 0$. Because the generators $\{\beta_{j,m}^* : 1\leq j \leq N_m\}$ are successive, any generator $j > 1$ must satisfy:
\[\min\{\log_2(s(\beta_{j,m}^*)),\min\supp(\beta_{j,m}^*)\} > \max \supp(\beta_{1,m}^*) \ge \min \text{ran}(x_{n_m}) > h_{m-1}\]
Thus, $\beta_{1,m}^*$ is the \emph{only} generator that could possibly have $\min \supp(\beta_{j,m}^*) \le h_{m-1}$ or $s(\beta_{j,m}^*) \le 2^{h_{m-1}}$.  Define $F_m = \{2,\ldots,N_m\}$ and $x_m^* = \sum_{j\in F_m} \lambda_{j,m} \alpha_{j,m}^*$. Clearly, \eqref{L3:item2} and \eqref{L3:item3} are satisfied. Furthermore, because $\|\alpha_{1,m}^*\|_1\leq 1$,
\[ \left|\alpha_{1,m}^*(x_{n_m}) \right| \le \|x_{n_m}\|_\infty< \epsilon_m.\]
Then $x_m^*(x_{n_m}) = y_m^*(x_{n_m}) - \lambda_{1,m} \alpha_{1,m}^*(x_{n_m}) > 1 - \epsilon_m$.
\end{proof}

\begin{remark}\label{R3}
A direct consequence of Lemma~\ref{L3} is that the sequence $(x_k^*)_k$ constructed above has the property that for every sequence of scalars $(\lambda_k)_{k=1}^m$ with $\sum_{k=1}^m \lambda_k^2 \le 1$, the linear combination
\[ x^* = \sum_{k=1}^m \lambda_k x_k^* = \sum_{k=1}^m \sum_{j \in F_k} (\lambda_k \lambda_{j,k}) \alpha_{j,k}^* \]
belongs to $G$. This holds because $\bigcup_{k=1}^m \{\beta_{j,k}^* : j \in F_k\}$ is a vfg family and 
\[\sum_{k=1}^m \sum_{j \in F_k} (\lambda_k \lambda_{j,k})^2 = \sum_{k=1}^m \lambda_k^2 \sum_{j \in F_k} \lambda_{j,k}^2 \le \sum_{k=1}^m \lambda_k^2 \le 1\]
\end{remark}

\begin{proposition}\label{P2}
Let $(x_n)_{n=1}^\infty$ be a normalized block sequence in $X_\alpha$ such that $\|x_n\|_\infty \to 0$ as $n \to \infty$. Then, for every $\epsilon > 0$, there exists a subsequence $(x_{n_k})_{k=1}^\infty$ that admits a $(1-\epsilon)$ lower $\ell_2$ estimate; i.e., for any sequence of scalars $(b_k)_{k=1}^m$, we have:
\begin{equation}
\left\| \sum_{k=1}^m b_k x_{n_k} \right\|_\alpha \ge (1 - \epsilon) \left( \sum_{k=1}^m b_k^2 \right)^{1/2}
\end{equation}
\end{proposition}

\begin{proof}
Fix $\epsilon > 0$. Apply Lemma~\ref{L3} with $\epsilon_k = \epsilon$ for all $k \in \mathbb{N}$ to obtain a subsequence $(x_{n_k})_k$ and a sequence of functionals $(x_k^*)_k \subset G$ satisfying conditions (1)--(3).

Let $(b_k)_{k=1}^m$ be an arbitrary finite sequence of scalars, and set $\lambda_k = \frac{b_k}{\left(\sum_{j=1}^m b_j^2\right)^{1/2}}$ so that $\sum_{k=1}^m \lambda_k^2 = 1$. By Remark~\ref{R3}, the functional $x^* = \sum_{k=1}^m \lambda_k x_k^*$ belongs to $G$.

Since $\text{supp}(x_k^*) \subset \text{ran}(x_{n_k})$ and $(x_{n_k})_k$ is a block sequence, $x_k^*(x_{n_l}) = 0$ for all $k \neq l$. Evaluating $x^*$ on $\sum_{k=1}^m b_k x_{n_k}$ yields:
\[ x^*\left( \sum_{k=1}^m b_k x_{n_k} \right) = \sum_{k=1}^m \lambda_k b_k x_k^*(x_{n_k}) > \sum_{k=1}^m \frac{b_k^2}{\left(\sum_{j=1}^m b_j^2\right)^{1/2}} (1 - \epsilon) = (1 - \epsilon) \left( \sum_{k=1}^m b_k^2 \right)^{1/2} \]
Taking the supremum over all elements in $G$ yields the required lower bound.
\end{proof}

\begin{lemma}\label{L9}
Let $x \in X_\alpha$ with finite support, $k \in \N$, and a level set $L_m = \{ p \in \N : |p| = m \}$ such that $k < \min L_m$ and $\max L_m < \min \supp(x)$. Then for every functional $\alpha^* = \beta^*|_I$ where $I \subset \ran(x)$ and $\beta^*$ is an average, there exists another average $\beta_1^*$ such that:
\begin{enumerate}
    \item[(i)] $\beta_1^*|_I = \alpha^*$ and $s(\beta_1^*) = s(\beta^*)$.
    \item[(ii)] $k < \min \supp(\beta_1^*)$.
\end{enumerate}
\end{lemma}

\begin{proof}
There are two cases to consider.

\textbf{Case 1:} $\beta^*$ is a comparable average ($\beta^* \in \mathcal{A}_c$). 
We set $k_0 = \min (I \cap \supp(\beta^*))$. 
If $\beta^*(k_0) > 0$, then we extend $\alpha^*$ to a comparable average $\beta_1^*$ of size $s(\beta_1^*) = s(\beta^*)$ by appending elements strictly greater than $\max I$. Because $k_0 \in I \subset \ran(x)$, we have $k < \min L_m < \min \supp(x) \le k_0 = \min \supp(\beta_1^*)$.
If $\beta^*(k_0) < 0$, we first add to $\alpha^*$ the unique node $p \in L_m$ such that $p \prec k_0$ (which exists and is unique by the tree structure). We then complete the new comparable average $\beta_1^*$ to size $s(\beta^*)$ by adding nodes strictly greater than $\max I$ as before. In this case, $\min \supp(\beta_1^*) = p \in L_m$, and since $k < \min L_m \le p$, we again have $k < \min \supp(\beta_1^*)$.

\textbf{Case 2:} $\beta^*$ is an incomparable average ($\beta^* \in \mathcal{A}_{inc}$). 
In this case, we substitute every node $n \in \supp(\beta^*)$, $ n < \min I$ with some $k_n \succ n$ such that $k_n > \max I$. Because the nodes of an incomparable average lie on distinct branches, extending them higher up their respective branches maintains their incomparability. The resulting functional $\beta_1^*$ is an incomparable average of the same size, its restriction to $I$ remains zero everywhere except where it coincides with $\alpha^*$, and all its supporting nodes are strictly bounded below by $\min I \ge \min \supp(x) > k$.

This completes the proof.
\end{proof}

\noindent\textbf{Notation.} For a fixed representation 
\[
x^*=\sum_{i\in F}\lambda_i\alpha_i^*\in G,
\qquad \alpha_i^*=\beta_i^*|_{I_i},
\qquad \sum_{i\in F}\lambda_i^2\leq 1,
\]
where $\{\beta_i^*\}_{i\in F}$ is a nonempty successive very fast growing family of averages, we define
\[
\min s(x^*)=\min_{i\in F}s(\beta_i^*).
\]
We also define
\[
\min\operatorname{supp\,gen}(x^*)=\min_{i\in F}\min\supp(\beta_i^*).
\]

This notation depends on the chosen representation, including its generators. Unless explicitly specified, a representation is understood to be fixed.

\begin{lemma}\label{L10}
Let $(x_n)_{n=1}^\infty$ be a normalized block sequence in $X_\alpha$ and let $\theta > 0$. Let also $(x_n^*)_{n=1}^\infty \subset G$ be a sequence of functionals such that for all $n \in \mathbb{N}$:
\begin{enumerate}
    \item $x_n^*(x_n) > \theta$
    \item $\min s(x_n^*) \to \infty$ as $n \to \infty$
\end{enumerate}
Then $(x_n)_{n=1}^\infty$ has a subsequence admitting a lower $\ell_2$ estimate with constant $\frac{\theta}{2}$.
\end{lemma}

\begin{proof}
We will construct a subsequence $(x_{n_k})_{k=1}^\infty$ and a sequence of functionals $(z_k^*)_{k=1}^\infty \subset G$ such that for all $k$, $z_k^*(x_{n_k}) > \frac{\theta}{2}$, $\text{supp}(z_k^*) \subset \text{ran}(x_{n_k})$, and the union of all generators of $(z_k^*)_{k=1}^\infty$ forms a successive very fast growing (vfg) family.

We proceed by induction. For $k=1$, set $n_1 = 1$. We define $y_1^* = x_{n_1}^*|_{\text{ran}(x_{n_1})} = \sum_{j \in F_1} \lambda_{j,1} \alpha_{j,1}^*$. We have $y_1^*(x_{n_1}) > \theta > \frac{\theta}{2}$. We set $z_1^* = y_1^* \in G$. Let 
$$h_1 = \max \supp(\text{generators of } z_1^*).$$

Assume that for $m > 1$, we have selected $n_1 < n_2 < \dots < n_{m-1}$ and functionals $z_1^*, \dots, z_{m-1}^* \subset G$ satisfying the inductive hypothesis. Let 
$$h_{m-1} = \max \supp(\text{generators of } z_{m-1}^*).$$
We also choose a $q_{m-1}\in \N $ such that 
\[ \min L_{q_{m-1}} = \min \{ k \in \N : |k| = q_{m-1}\} > 2^{h_{m-1}} \]

Because $(x_n)_{n=1}^\infty$ is a block sequence and $\min s(x_n^*) \to \infty$, we can choose $n_m > n_{m-1}$ sufficiently large such that:
\begin{enumerate}
    \item[(i)] $\min \text{ran}(x_{n_m}) > \max L_{q_{m-1}}$
    \item[(ii)] $\min s(x_{n_m}^*) > 2^{h_{m-1}}$
\end{enumerate}

Let $y_m^* = x_{n_m}^*|_{\text{ran}(x_{n_m})} = \sum_{j=1}^{r_m} \lambda_{j,m} \alpha_{j,m}^* \in G$, where the generators $\beta_{1,m}^*, \dots, \beta_{r_m,m}^*$ are successive. We know that $y_m^*(x_{n_m}) > \theta$. We decompose this evaluation into the first generator and the rest:
\begin{equation}
y_m^*(x_{n_m}) = \lambda_{1,m} \alpha_{1,m}^*(x_{n_m}) + \sum_{j=2}^{r_m} \lambda_{j,m} \alpha_{j,m}^*(x_{n_m}) > \theta.
\end{equation}
Since the sum is strictly greater than $\theta$, we must have one of two cases:

{ \bf 1:} $\sum_{j=2}^{r_m} \lambda_{j,m} \alpha_{j,m}^*(x_{n_m}) > \frac{\theta}{2}$.

{ \bf 2:} $\lambda_{1,m} \alpha_{1,m}^*(x_{n_m}) > \frac{\theta}{2}$.

We construct the functional $z_m^*$ depending on which case holds:

\textbf{ Case 1:}
Assume Case 1 holds. Since $\alpha_{1,m}^*$ has non-empty support on $\text{ran}(x_{n_m})$, we have $\max \supp(\beta_{1,m}^*) \ge \min \text{ran}(x_{n_m}) > h_{m-1}$. Because the generators are successive, this implies that for all $j \ge 2$, $\min \supp(\beta_{j,m}^*) > \max \supp(\beta_{1,m}^*) > h_{m-1}$.
We define $z_m^* = \sum_{j=2}^{r_m} \lambda_{j,m} \alpha_{j,m}^* \in G$. Then $z_m^*(x_{n_m}) > \frac{\theta}{2}$, and all generators of $z_m^*$ are strictly successive to $h_{m-1}$ with $s(\beta_{2, m}^*)> 2^{h_{m-1}}$ which yields that 
the family $\{ \beta_{j, m}^* \}_{j=2}^{r_m}$ extends the generators associated to $\{ z_k^*\}_{k < m} $ to a vfg family.

\textbf{ Case 2:}
Assume Case 2 holds. Then from Lemma \ref{L9} there exists a generator $\gamma^* $ such that:
\begin{enumerate}
\item $ \min supp(\gamma^*) >2^{h_{m- 1} }$
\item If $ \alpha_{1, m}^* = \beta_{1, m}^*|_I  \quad \text{then} \quad \gamma|_I = \alpha_{1, m}^*$
\end{enumerate}
It is clear that $\gamma^*$ the generator of $\alpha_{1, m}^*$ extends the generators associated to $\{ z_k^*\}_{k < m} $
to  a vfg family.

In both cases, we produce $z_m^* \in G$ satisfying $z_m^*(x_{n_m}) > \frac{\theta}{2}$ whose generators are strictly successive to $h_{m-1}$ and whose sizes are strictly $> 2^{h_{m-1}}$ (by condition ii). This completes the inductive step.

By passing to a subsequence if necessary, we obtain the block sequence $(x_{n_k})_{k=1}^\infty$ and the corresponding sequence $(z_k^*)_{k=1}^\infty$. Let $(b_k)_{k=1}^M$ be any finite sequence of scalars. We set $\lambda_k = \frac{b_k}{\left(\sum_{j=1}^M b_j^2\right)^{1/2}}$ so that $\sum_{k=1}^M \lambda_k^2 = 1$. 

The functional $x^* = \sum_{k=1}^M \lambda_k z_k^*$ belongs to $G$ because its combined generators form a successive vfg family, and the sum of squares of all coefficients is bounded by $\sum_{k=1}^M \lambda_k^2 \le 1$.
Because the supports of $z_k^*$ are disjoint and restricted to $\text{ran}(x_{n_k})$, $z_k^*(x_{n_l}) = 0$ for $k \neq l$. Evaluating $x^*$ on $x = \sum_{k=1}^M b_k x_{n_k}$ yields:
\begin{equation}
x^*(x) = \sum_{k=1}^M \lambda_k b_k z_k^*(x_{n_k}) > \sum_{k=1}^M \frac{b_k^2}{\left(\sum_{j=1}^M b_j^2\right)^{1/2}} \left(\frac{\theta}{2}\right) = \frac{\theta}{2} \left( \sum_{k=1}^M b_k^2 \right)^{1/2}
\end{equation}
Taking the supremum over $x^* \in G$ concludes the proof:
\begin{equation}
\left\| \sum_{k=1}^M b_k x_{n_k} \right\|_\alpha \ge \frac{\theta}{2} \left( \sum_{k=1}^M b_k^2 \right)^{1/2}
\end{equation}
\end{proof}

\subsection{$c_0$  Spreading Models in $X_\alpha$}

We start by recalling the definitions of $c_0$ and $\ell_1$ spreading models.

\begin{definition} 
A normalized sequence $(x_n)_n$ in a Banach space $X$ with a basis $(e_n)_n$ is:
\begin{enumerate}
\item a $c_0$ spreading model with constant $C \ge 1$ if for every $H \subset \N$ with $\# H \le \min H$ we have 
 \[ \left\| \sum_{n\in H} x_n \right\| \le C \]
\item an $\ell_1$ spreading model with constant $C \ge 1$ if for every $H \subset \N$ with $\# H \le \min H$ we have 
 \[ \sum_{n\in H} |\lambda_n| \le C \left\| \sum_{n\in H} \lambda_n x_n \right\| \]
\end{enumerate}
\end{definition}

\begin{remark}\label{R4}
Let us point out, for later use, that if $(x_n)_n$ is a normalized sequence in $X$ which is a $c_0$ spreading model and $(x_n^*)_n$ is a normalized sequence in $X^*$ such that $x_n^*(x_n) \ge \epsilon > 0$, and $x_n^*(x_m) = 0$ for $n \ne m$, then $(x_n^*)_n$ is an $\ell_1$ spreading model.
\end{remark}

\begin{proposition}\label{P4} 
The basis $(e_n)_n$ is a $c_0$ spreading model with a constant $C < 2 $. 
\end{proposition}

\begin{proof}
Let $H \subset \N$ be a finite subset such that $m = \# H \le \min H$. We consider the vector $x = \sum_{n \in H} e_n$. Note that $\|x\|_\infty = 1$. Since $G$ contains the singletons $\{\pm e_n^*\}_{n \in T}$, we easily see that $\|e_n\|_\alpha = 1$ for all $n$, so the basis is normalized. We will show that $\|x\|_\alpha \leq 2/\sqrt{3}<2$.

Let $x^* = \sum_{i=1}^l \lambda_i \alpha_i^* \in G$ be an arbitrary functional defined by a successive very fast growing (vfg) family of generators $\{\beta_i^*\}_{i=1}^l$, with $\sum_{i=1}^l \lambda_i^2 \le 1$. Assume, without loss of generality, that $\alpha_1^*(x)\neq 0$.

\textbf{Step 1: Bounding the first average.}
Because $\|\alpha_{1}^*\|_1 \le 1$, evaluating on $x$ yields:
\begin{equation}
|\alpha_{1}^*(x)| \le \|\alpha_{1}^*\|_1 \|x\|_\infty \le 1.
\end{equation}

\textbf{Step 2: Bounding the remaining averages.}
For any $i > 1$, the functional $\alpha_i^*$ is evaluated on at most $\# H = m$ nodes. Thus:
\begin{equation}
|\alpha_i^*(x)| \le \frac{m}{s(\beta_i^*)} \leq \frac{m}{2^{h_{i-1}}}.
\end{equation}

\textbf{Conclusion.}
Applying the Cauchy-Schwarz inequality:
\begin{equation}
|x^*(x)|^2 \le \Big(|\alpha_{1}^*(x)|^2 + \sum_{i > 1}|\alpha_i^*(x)|^2\Big) \leq \Big(1 + m^2\sum_{i\geq m}\frac{1}{4^i}\Big) = \frac{3\cdot 4^{m-1}+m^2}{3\cdot4^{m-1}} \leq \frac{4}{3}.
\end{equation}
We conclude $\|x\|_\alpha \le 2/\sqrt{3}$.
\end{proof}

\subsection{Non-embedding of $c_0$ into $X_\alpha$}

We now use the lower $\ell_2$ estimate established in Proposition~\ref{P2} alongside Ramsey's Theorem on the tree structure to prove that $X_\alpha$ contains no isomorphic copy of $c_0$.

\begin{theorem}
The space $c_0$ does not embed into $X_\alpha$.
\end{theorem}

\begin{proof}
Suppose, towards a contradiction, that $c_0$ embeds into $X_\alpha$. Then there exists a normalized sequence $(y_n)_{n=1}^\infty$ in $X_\alpha$ that is equivalent to the standard unit vector basis of $c_0$. Because the $c_0$ basis is weakly null and the basis $(e_n)_{n \in T}$ of $X_\alpha$ is shrinking, the sequence $(y_n)_{n=1}^\infty$ is weakly null. By standard perturbation arguments (e.g., the Bessaga-Pe\l{}czy\'nski selection principle), we may pass to a subsequence and assume without loss of generality that there is a normalized block sequence $(x_n)_{n=1}^\infty$ equivalent to the $c_0$ basis.

Since $(x_n)_{n=1}^\infty$ is equivalent to the $c_0$ basis, there exists a constant $C > 0$ such that for every finite subset $H \subset \mathbb{N}$, we have:
\begin{equation}\label{eq_c0}
\left\| \sum_{n \in H} x_n \right\|_\alpha \le C
\end{equation}

We claim that $\liminf_{n \to \infty} \|x_n\|_\infty > 0$. If this were not true, we could extract a subsequence $(x_{n_k})_{k=1}^\infty$ such that $\|x_{n_k}\|_\infty \to 0$. By Proposition~\ref{P2}, we could extract a further subsequence admitting a lower $\ell_2$ estimate. However, no subsequence of the $c_0$ basis admits a lower $\ell_2$ estimate, which is a contradiction. Thus, there exists $\theta > 0$ such that $\|x_n\|_\infty > \theta$ for all $n$.

By changing the signs of $x_n$ if necessary, we can choose for each $n$ a node $k_n \in \supp(x_n)$ such that $x_n(k_n) > \theta$. Since $(x_n)_{n=1}^\infty$ is a block sequence, the nodes $(k_n)_{n=1}^\infty$ are distinct.

We apply Ramsey's Theorem to the partial order $\preceq$ of the dyadic tree $T$. The infinite set of nodes $\{k_n : n \in \mathbb{N}\}$ must contain an infinite sequence that is either a chain (pairwise comparable elements) or an antichain (pairwise incomparable elements). Passing to a subsequence, we assume $(k_n)_{n=1}^\infty$ is either purely incomparable or purely comparable. Let $l \in \mathbb{N}$ be an arbitrarily large integer.

\textbf{Case 1: The nodes $(k_n)_{n=1}^\infty$ are pairwise incomparable.}
We construct a successive very fast growing family of incomparable averages $\beta_1^*, \dots, \beta_l^*$. 
Set $s_1 = 8$. Choose $s_1$ indices to form a block $H_1$, and define $\beta_1^* = \frac{1}{s_1} \sum_{n \in H_1} e_{k_n}^* \in \mathcal{A}_{inc}$. We have $s(\beta_1^*) = s_1 \ge 8$.
Assuming $\beta_{i-1}^*$ is constructed, let $h_{i-1} = \max \supp(\beta_{i-1}^*)$. Set $s_i = 2^{h_{i-1}} + 1$, select the next $s_i$ indices to form $H_i$, and define $\beta_i^* = \frac{1}{s_i} \sum_{n \in H_i} e_{k_n}^* \in \mathcal{A}_{inc}$. 
This forms a valid successive vfg family. 

Let $H = \bigcup_{i=1}^l H_i$. For each $i$, the evaluation gives:
\[ \beta_i^* \left( \sum_{n \in H} x_n \right) = \frac{1}{s_i} \sum_{n \in H_i} x_n(k_n) > \frac{1}{s_i} (s_i \cdot \theta) = \theta \]
Define $x^* = \frac{1}{\sqrt{l}} \sum_{i=1}^l \beta_i^*$. Since $\sum_{i=1}^l \left(\frac{1}{\sqrt{l}}\right)^2 = 1$, $x^* \in G$. Evaluating on $x = \sum_{n \in H} x_n$, we obtain:
\[ \|x\|_\alpha \ge x^*(x) = \frac{1}{\sqrt{l}} \sum_{i=1}^l \beta_i^*(x) > \frac{1}{\sqrt{l}} (l \cdot \theta) = \theta \sqrt{l} \]

\textbf{Case 2: The nodes $(k_n)_{n=1}^\infty$ are pairwise comparable.}
In this case, the nodes lie on a single branch $\sigma$. Because $(x_n)_n$ is a successive block sequence, we may assume, passing if necessary to subsequence, that for every n there exists a $t_n \in \s$ with  $\max \supp(x_n) < t_n <\min \supp(x_{n+1})$ . Since $t_n \notin \supp(x_m)$ for all $m$, we have $x_m(t_n) = 0$. 
This gives an alternating sequence on the branch $\sigma$: $k_1 \prec t_1 \prec k_2 \prec t_2 \dots$

We construct a successive vfg family of comparable averages $\beta_1^*, \dots, \beta_l^*$ similarly. 
Set $s_1 = 8$. Choose $s_1$ indices to form $H_1$, and define $\beta_1^* = \frac{1}{2s_1} \sum_{n \in H_1} (e_{k_n}^* - e_{t_n}^*) \in \mathcal{A}_c$. 
Assuming $\beta_{i-1}^*$ is constructed, let $h_{i-1} = \max \supp(\beta_{i-1}^*)$. Set $s_i = 2^{h_{i-1}} + 1$, select the next $s_i$ indices to form $H_i$, and define $\beta_i^* = \frac{1}{2s_i} \sum_{n \in H_i} (e_{k_n}^* - e_{t_n}^*) \in \mathcal{A}_c$.

Let $H = \bigcup_{i=1}^l H_i$. For each $i$, the evaluation gives:
\[ \beta_i^* \left( \sum_{n \in H} x_n \right) = \frac{1}{2s_i} \sum_{n \in H_i} (x_n(k_n) - x_n(t_n)) > \frac{1}{2s_i} (s_i \cdot \theta - 0) = \frac{\theta}{2} \]
Define $x^* = \frac{1}{\sqrt{l}} \sum_{i=1}^l \beta_i^* \in G$. Evaluating on $x = \sum_{n \in H} x_n$, we obtain:
\[ \|x\|_\alpha \ge x^*(x) = \frac{1}{\sqrt{l}} \sum_{i=1}^l \beta_i^*(x) > \frac{1}{\sqrt{l}} \left(l \cdot \frac{\theta}{2}\right) = \frac{\theta}{2} \sqrt{l} \]

In both cases, for sufficiently large $l$, the norm $\left\| \sum_{n \in H} x_n \right\|_\alpha$ exceeds the bound $C$ from Equation \ref{eq_c0}. This is a contradiction. Thus, $c_0$ does not embed into $X_\alpha$.
\end{proof}

\subsection{$\ell_2$-Saturation of $X_\alpha$}

With the upper and lower $\ell_2$ estimates established, and having proven that $c_0$ does not embed into $X_\alpha$, we can now prove that $X_\alpha$ is $\ell_2$-saturated. Actually, we will prove a stronger result: the copy of $\ell_2$ can be chosen to be complemented. We first establish a useful functional evaluation lemma.

\begin{lemma}\label{L_comp}
Let $\{x_n^*\}_{n=1}^m \subset G$ be a finite sequence of functionals such that for every sequence of scalars $\{\lambda_n\}_{n=1}^m$ with $\sum_{n=1}^m \lambda_n^2 \le 1$, the functional $\sum_{n=1}^m \lambda_n x_n^*$ belongs to $G$. Then for every $x \in X_\alpha$, we have:
\begin{equation}
\left( \sum_{n=1}^m |x_n^*(x)|^2 \right)^{1/2} \le \|x\|_\alpha
\end{equation}
\end{lemma}

\begin{proof}
Let $x \in X_\alpha$. If $x_n^*(x) = 0$ for all $n$, the inequality is trivial. Otherwise, choose $\lambda_n = \frac{x_n^*(x)}{\left( \sum_{j=1}^m |x_j^*(x)|^2 \right)^{1/2}}$. By definition, $\sum_{n=1}^m \lambda_n^2 = 1$. By our hypothesis, the functional $x^* = \sum_{n=1}^m \lambda_n x_n^*$ belongs to $G$.
Evaluating $x^*$ on $x$, we obtain:
\[ x^*(x) = \sum_{n=1}^m \lambda_n x_n^*(x) = \frac{\sum_{n=1}^m |x_n^*(x)|^2}{\left( \sum_{j=1}^m |x_j^*(x)|^2 \right)^{1/2}} = \left( \sum_{n=1}^m |x_n^*(x)|^2 \right)^{1/2} \]
Since $x^* \in G$, we have $x^*(x) \le \|x\|_\alpha$, which completes the proof.
\end{proof}

\begin{theorem}\label{thm_complemented_l2_Xalpha}
Every infinite-dimensional closed subspace of $X_\alpha$ contains a subspace isomorphic to $\ell_2$ which is complemented in $X_\alpha$.
\end{theorem}

\begin{proof}
Let $Y$ be an infinite-dimensional closed subspace of $X_\alpha$. Since the basis $(e_n)_n$ of $X_\alpha$ is shrinking, every infinite-dimensional closed subspace contains a normalized weakly null sequence. By the Bessaga-Pe\l{}czy\'nski selection principle, we can pass to a subsequence and assume, via a small perturbation, that $Y$ contains a normalized block sequence $(x_n)_{n=1}^\infty$.

We claim that the subspace spanned by $(x_n)_{n=1}^\infty$ contains a further normalized block sequence $(y_n)_{n=1}^\infty$ such that $\|y_n\|_\infty \to 0$ as $n \to \infty$. 
Suppose, towards a contradiction, that this is false. Then there exists a constant $\theta > 0$ such that for all normalized finite block vectors $y \in [x_n]$, we have $\|y\|_\infty \ge \theta$. Since $\|y\|_\infty \le \|y\|_\alpha = 1$, the norms $\|\cdot\|_\alpha$ and $\|\cdot\|_\infty$ are equivalent on the infinite-dimensional subspace generated by $(x_n)_{n=1}^\infty$. 
This implies that $[x_n]$ embeds isomorphically into $c_0$. Since every infinite-dimensional subspace of $c_0$ contains an isomorphic copy of $c_0$, this would force $c_0$ to embed into $X_\alpha$. But we proved in the previous Theorem that $c_0$ does not embed into $X_\alpha$, which is a contradiction.

Thus, we can extract a normalized block sequence $(y_n)_{n=1}^\infty$ such that $\|y_n\|_\infty \to 0$.

Fix $0 < \epsilon < 1/2$. By Lemma~\ref{L3}, there exists a subsequence $(y_{n_k})_{k=1}^\infty$ and a sequence of functionals $(x_k^*)_{k=1}^\infty \subset G$ such that:
\begin{enumerate}
    \item For every $k \in \mathbb{N}$, $x_k^*(y_{n_k}) > 1 - \epsilon > 1/2$ and $\text{supp}(x_k^*) \subset \text{ran}(y_{n_k})$.
    \item By Remark~\ref{R3}, for every finite sequence of scalars $(\lambda_k)_k$ with $\sum_{k=1}^m \lambda_k^2 \le 1$, we have $\sum_{k=1}^m \lambda_k x_k^* \in G$.
\end{enumerate}

Because $\text{supp}(x_k^*) \subset \text{ran}(y_{n_k})$ and $(y_{n_k})_k$ is a block sequence, the functionals are biorthogonal to the sequence in the sense that $x_k^*(y_{n_l}) = 0$ for $k \neq l$.

As shown in Proposition~\ref{P2} and the upper $\ell_2$ estimate, the sequence $(y_{n_k})_{k=1}^\infty$ is equivalent to the standard unit vector basis of $\ell_2$. Specifically, there exist constants $c, C > 0$ such that for any finite sequence of scalars $(b_k)_{k=1}^m$:
\begin{equation}
c \left( \sum_{k=1}^m b_k^2 \right)^{1/2} \le \left\| \sum_{k=1}^m b_k y_{n_k} \right\|_\alpha \le C \left( \sum_{k=1}^m b_k^2 \right)^{1/2}
\end{equation}

We define $z_k = \frac{1}{x_k^*(y_{n_k})} y_{n_k}$. Since $1 \le \frac{1}{x_k^*(y_{n_k})} < 2$, the sequence $(z_k)_{k=1}^\infty$ is also equivalent to the $\ell_2$ basis, and it spans the same subspace as $(y_{n_k})_{k=1}^\infty$. The biorthogonality relation now reads $x_k^*(z_l) = \delta_{k,l}$.

We define a projection $P: X_\alpha \to [z_k]_{k=1}^\infty$ by:
\begin{equation}
P(x) = \sum_{k=1}^\infty x_k^*(x) z_k
\end{equation}
To see that $P$ is bounded, let $x \in X_\alpha$. Using the upper $\ell_2$ estimate of $(z_k)_{k=1}^\infty$ (with constant $\tilde{C} = 2C$) and the newly established Lemma~\ref{L_comp}:
\begin{equation}
\|P(x)\|_\alpha = \left\| \sum_{k=1}^\infty x_k^*(x) z_k \right\|_\alpha \le \tilde{C} \left( \sum_{k=1}^\infty |x_k^*(x)|^2 \right)^{1/2} \le \tilde{C} \|x\|_\alpha
\end{equation}
Therefore, $P$ is a bounded linear operator. Furthermore, for any $l \in \mathbb{N}$, $P(z_l) = \sum_{k=1}^\infty x_k^*(z_l) z_k = z_l$. Thus, $P$ is a bounded projection onto the closed linear span of $(z_k)_{k=1}^\infty$. 

Since $[z_k]_{k=1}^\infty = [y_{n_k}]_{k=1}^\infty \subset Y$ and is isomorphic to $\ell_2$, we conclude that $Y$ contains a complemented subspace isomorphic to $\ell_2$.
\end{proof}

\subsection{The Main Dichotomy}

To finish the study of the space $X_\alpha$, we present a structural dichotomy for normalized block sequences. We introduce an index that measures whether the dual functionals corresponding to the block sequence admit boundedly small building blocks.

\begin{definition}[The $\alpha$-index]\label{def_alpha_index}
Let $(x_n)_{n=1}^\infty$ be a normalized block sequence in $X_\alpha$. We say that $\alpha((x_n)_n) > 0$ if there exists a real number $\theta > 0$, a subsequence $(x_{n_k})_{k=1}^\infty$, and a sequence of functionals $(x_k^*)_{k=1}^\infty \subset G$ such that for all $k \in \mathbb{N}$:
\begin{enumerate}
    \item $k < \min s(x_k^*)$, and
    \item $x_k^*(x_{n_k}) > \theta$.
\end{enumerate}
Otherwise, we say that $\alpha((x_n)_n) = 0$.
\end{definition}

The vanishing of the $\alpha$-index means that it is impossible to evaluate elements of the block sequence with arbitrarily large index significantly away from zero with functionals whose underlying generators have arbitrarily large sizes. 

\begin{lemma}\label{L_alpha_zero}
Let $(x_n)_{n=1}^\infty$ be a normalized block sequence in $X_\alpha$. If $\alpha((x_n)_n) = 0$, then for every $\epsilon > 0$ there exist integers $N_\epsilon$ and $L_\epsilon$ such that for every $n \ge N_\epsilon$ and every functional $x^* \in G$ satisfying $L_\epsilon \le \min s(x^*)$, we have:
\begin{equation}
|x^*(x_n)| < \epsilon
\end{equation}
\end{lemma}

\begin{proof}
Suppose, towards a contradiction, that there exists $\epsilon > 0$ such that for every pair of integers $(N, L)$, there is an index $n \ge N$ and a functional $x^* \in G$ with $\min s(x^*) \ge L$ such that $x^*(x_n) \ge \epsilon$. 
We can inductively construct a subsequence $(x_{n_k})_{k=1}^\infty$ and a sequence of functionals $(x_k^*)_{k=1}^\infty \subset G$. 
For $k=1$, set $N=1, L=2$, and find $n_1 \ge 1$ and $x_1^* \in G$ such that $\min s(x_1^*) \ge 2 > 1$ and $x_1^*(x_{n_1}) \ge \epsilon$. 
Assuming we have chosen $(n_1, x_1^*), \dots, (n_{k-1}, x_{k-1}^*)$, we set $N = n_{k-1} + 1$ and $L = k + 1$. By our assumption, there exist $n_k \ge n_{k-1} + 1$ and $x_k^* \in G$ such that $\min s(x_k^*) \ge k+1 > k$ and $x_k^*(x_{n_k}) \ge \epsilon$. 
This construction yields $\theta = \epsilon / 2 > 0$, a subsequence $(x_{n_k})_{k=1}^\infty$, and $(x_k^*)_{k=1}^\infty \subset G$ satisfying $k < \min s(x_k^*)$ and $x_k^*(x_{n_k}) \geq \epsilon > \theta$. But this implies $\alpha((x_n)_n) > 0$, which contradicts the hypothesis.
\end{proof}

\begin{theorem}\label{Main_Dichotomy}
For every normalized block sequence $(x_n)_{n=1}^\infty$ in $X_\alpha$, one of the following holds:
\begin{enumerate}
    \item There exists a subsequence $(x_{n_k})_{k=1}^\infty$ which is equivalent to the standard unit vector basis of $\ell_2$, and the closed linear span $[x_{n_k}]_{k=1}^\infty$ is complemented in $X_\alpha$.
    \item There exists a subsequence $(x_{n_k})_{k=1}^\infty$ which is a $c_0$ spreading model.
\end{enumerate}
\end{theorem}

\begin{proof}
Let $(x_n)_{n=1}^\infty$ be a normalized block sequence. We consider the two mutually exclusive possibilities for its $\alpha$-index.

\textbf{Case 1: $\alpha((x_n)_n) > 0$}
By definition, there exist $\theta > 0$, a subsequence (which we will relabel as $(x_n)_{n=1}^\infty$ for simplicity), and $(x_n^*)_{n=1}^\infty \subset G$ such that $x_n^*(x_n) > \theta$ and $n < \min s(x_n^*)$ for all $n$. 
Because $n < \min s(x_n^*)$, we have $\min s(x_n^*) \to \infty$ as $n \to \infty$. 
This places us exactly in the hypothesis framework of Lemma~\ref{L10}. By Lemma~\ref{L10}, there exists a further subsequence $(x_{n_k})_{k=1}^\infty$ admitting a lower $\ell_2$ estimate with constant $\theta/2$, constructed via a biorthogonal sequence of functionals $(z_k^*)_{k=1}^\infty \subset G$ whose generators form a successive vfg family.
By the upper $\ell_2$ estimate (Proposition 3.3), $(x_{n_k})_{k=1}^\infty$ is therefore equivalent to the standard basis of $\ell_2$. Moreover, the successive vfg structure of the generators for $(z_k^*)_{k=1}^\infty$ guarantees that linear combinations satisfy the hypotheses of Lemma~\ref{L_comp}. Consequently, the projection $P(x) = \sum_{k=1}^\infty \frac{z_k^*(x)}{z_k^*(x_{n_k})} x_{n_k}$ is bounded, and the generated space $[x_{n_k}]_{k=1}^\infty$ is complemented in $X_\alpha$.

\textbf{Case 2: $\alpha((x_n)_n) = 0$}
We will extract a subsequence that is a $c_0$ spreading model. We inductively construct a sequence of indices $(n_k)_{k=1}^\infty$ such that the resulting block sequence forces the evaluation of any vfg family on it to be small. 
For $k=1$, we choose any $n_1 \ge 1$. 
Assume we have chosen $n_1 < \dots < n_{k-1}$. Let $h_{k-1} = \max \text{ran}(x_{n_{k-1}})$. We define the precision level $\epsilon_k = 1/k^2$. By Lemma~\ref{L_alpha_zero}, there exist parameters $N_{\epsilon_k}$ and $L_{\epsilon_k}$ controlling the dual evaluation.
Since $2^d \to \infty$, there exists an integer $d_k$ such that $2^{d_k} > L_{\epsilon_k}$. We require the next block to be located sufficiently far down the tree. Specifically, we choose $n_k > n_{k-1}\vee N_{\epsilon_k}$ such that:
\begin{equation}
\min \supp(x_{n_k}) > \max \{d_k, h_{k-1}\}.
\end{equation}

We claim that the subsequence $(y_k)_{k=1}^\infty = (x_{n_k})_{k=1}^\infty$ is a $c_0$ spreading model. 
Let $H \subset \mathbb{N}$ be a finite subset with $2\leq m = \#H \le \min H$, and let $x = \sum_{k \in H} y_k$. Let $x^* = \sum_{i=1}^l \lambda_i \alpha_i^* \in G$ be defined by a successive vfg family $\{\beta_i^*\}_{i=1}^l$. Assume, without loss of generality, that, for $1\leq i\le l$, $\alpha_i^*(x)\neq 0$. As in the proof of Proposition~\ref{P4},
\[|\alpha_1^*(x)| \leq 1.\]
For any $i > 1$, the vfg property dictates that
\[s(\beta_i^*) > 2^{\max \supp(\beta_{i-1}^*)} \ge 2^{\min \text{ran}(y_{\min H})} \ge 2^{\min \supp(x_{n_m})}>2^{d_m}>L_{\epsilon_m}.\]
Since, for $k\in H$, $n_k\geq N_{\epsilon_m}$:
\[\big|\sum_{i\geq 2}\lambda_i\alpha_i^*(x)\big| \leq m\varepsilon_m = 1/m.\]
We deduce $|x^*(x)| \leq 1+1/m$, which implies that $(y_k)_{k=1}^\infty$ is a $c_0$ spreading model.
\end{proof}

\section{The Line Space Variant and Quasi-Reflexivity}
Here, we analyze a sequential variant of our space, $X_c$, constructed on $\N$ or by restricting the index set to a single branch and utilizing only comparable averages. We prove that this line space variant is quasi-reflexive of order one, closely mirroring the classical James space.

By restricting the index set strictly to the natural numbers $\N$ (viewed as a single branch without incomparable nodes) and eliminating the incomparable averages, the branching structure collapses. We are left with a sequential space controlled entirely by the variation of its elements, perfectly mirroring the classical James space $J$.

Let $X_c$ be the completion of $c_{00}(\N)$ equipped with the norm $\|\cdot\|_c$ generated exactly as $\|\cdot\|_\alpha$, but using \emph{only} the comparable averages $\mathcal{A}_c$. Since the elimination of $\mathcal{A}_{inc}$ does not affect the uniform boundedness of segment sums or the upper and lower $\ell_2$ estimates for comparable blocks, the basis $(e_n)_{n=1}^\infty$ is shrinking in $X_c$ (Corollary \ref{C2}).

\begin{theorem}
The space $X_c$ is quasi-reflexive of order one; namely, $\dim(X_c^{**} / X_c) = 1$.
\end{theorem}

\begin{proof}
\textbf{Step 1: Identification of the Bidual.}
Since the basis $(e_n)_{n=1}^\infty$ is shrinking, $X_c^*$ is canonically isomorphic to the closed linear span of the biorthogonal functionals $(e_n^*)_{n=1}^\infty$. We can therefore identify the bidual $X_c^{**}$ with the space of all scalar sequences $x^{**} = (a_n)_{n=1}^\infty$ whose partial sums are uniformly bounded in $X_c$, i.e.,
$$\sup_m \left\| \sum_{n=1}^m a_n e_n \right\|_c \le M < \infty$$
 The action of $x^{**}$ on any $x^* \in c_{00}(\N)^*$ is given by $x^{**}(x^*) = \lim_{m \to \infty} x^*\left(\sum_{n=1}^m a_n e_n\right)$.

\textbf{Step 2: Coordinate Convergence in $X_c^{**}$.}
We claim that for every $x^{**} = (a_n) \in X_c^{**}$, the limit $c = \lim_{n \to \infty} a_n$ exists. 
Suppose, towards a contradiction, that $(a_n)$ does not converge. Then $\limsup_{n \to \infty} a_n > \liminf_{n \to \infty} a_n$. Let $R$ and $r$ be real numbers such that $\limsup a_n > R > r > \liminf a_n$, and define $\delta = R - r > 0$. We can recursively select an infinite sequence of indices $n_1 < m_1 < n_2 < m_2 < \dots$ such that $a_{n_i} > R$ and $a_{m_i} < r$. This guarantees that $a_{n_i} - a_{m_i} > \delta$ for all $i$.

Let $l \in \N$ be an arbitrarily large integer. We construct a successive vfg family of comparable averages $\beta_1^*, \dots, \beta_l^* \in \mathcal{A}_c$ using these indices:
\begin{itemize}
    \item Set $s_1 = 8$. Choose the first $s_1/2 = 4$ pairs of indices to form $\beta_1^*$. Evaluating against $x^{**}$ yields:
    \[ \beta_1^*(x^{**}) = \frac{1}{s_1} \sum_{i=1}^4 (a_{n_i} - a_{m_i}) > \frac{4}{8}\delta = \frac{\delta}{2} \]
    \item Assuming $\beta_{j-1}^*$ is constructed, let $h_{j-1} = \max \supp(\beta_{j-1}^*)$. Set $s_j$ to be the smallest even integer strictly greater than $2^{h_{j-1}}$. Form $\beta_j^*$ using the next $s_j/2$ pairs of indices strictly after $h_{j-1}$. Again, $\beta_j^*(x^{**}) > \frac{\delta}{2}$.
\end{itemize}

Define the functional $x^* = \frac{1}{\sqrt{l}} \sum_{j=1}^l \beta_j^*$. Because its generators form a vfg family and the coefficients square-sum to $1$, $x^*$ belongs to the norming set $G$ restricted to $X_c$. Evaluating $x^*$ on $x^{**}$ yields:
\[ x^{**}(x^*) = \frac{1}{\sqrt{l}} \sum_{j=1}^l \beta_j^*(x^{**}) > \frac{1}{\sqrt{l}} \left(l \cdot \frac{\delta}{2}\right) = \frac{\delta}{2} \sqrt{l} \]
Since $x^{**} \in X_c^{**}$, we must have $x^{**}(x^*) \le M$. But $l$ can be arbitrarily large, which is a contradiction. Thus, $c = \lim_{n \to \infty} a_n$ exists.

\textbf{Step 3: The Constant Sequence.}
Let $z^{**} = (1, 1, 1, \dots)$ be the constant sequence of ones. We claim $z^{**} \in X_c^{**}$. 
For any comparable average $\alpha_c^* \in \mathcal{A}_c$, the definition of the functional evaluates identically to zero on $z^{**}$, and for any interval $I$ of $\mathbb{N}$,
\[ \big|\alpha_c^*|_I(z^{**})\big| \leq \frac{1}{2k} = \frac{1}{s(\alpha_c^*)}.\]
Therefore, for any functional $x^* \in G$ defined by a linear combination of comparable averages, $|x^*(z^{**})| \leq (1+\sum_{i\geq 1}4^{-i})^{1/2} < 2$. Thus, $z^{**}$ is a well-defined element of $X_c^{**} \setminus X_c$.

\textbf{Step 4: Decomposition and Conclusion.}
For any $x^{**} = (a_n) \in X_c^{**}$, let $c = \lim_{n \to \infty} a_n$. We can decompose $x^{**}$ as:
\[ x^{**} = c z^{**} + y^{**} \]
where $y^{**} = (a_n - c)$. Since $X_c^{**}$ is a vector space, $y^{**} \in X_c^{**}$ with norm bounded by some $M_y$, and by construction, $\lim_{n \to \infty} y^{**}(n) = 0$.

We must show that $y^{**}$ actually belongs to $X_c$. Let $S_m = \sum_{n=1}^m y^{**}(n) e_n$ be the partial sums of $y^{**}$. 
If $(S_m)$ does not converge in norm, it is not Cauchy. Then there exists $\epsilon > 0$ and an increasing sequence of indices $p_1 < q_1 < p_2 < q_2 < \dots$ such that the blocks $u_k = S_{q_k} - S_{p_k}$ satisfy $\|u_k\|_c \ge \epsilon$. Since the coordinates $y^{**}(n) \to 0$, we have $\|u_k\|_\infty \to 0$. 

Apply Lemma~\ref{L3}, whose proof applies unchanged to $X_c$, to the normalized block sequence $(u_k/\|u_k\|_c)_k$, whose $\|\cdot\|_\infty$ norms tend to zero. Passing to a subsequence and relabelling, we obtain functionals $x_k^*\in G$ supported on $(p_k,q_k]$ such that $x_k^*(u_k)>\frac12\|u_k\|_c$ and the combined generators of $(x_k^*)_k$ form a successive vfg family. Decreasing $\epsilon$ if necessary, we therefore have $x_k^*(y^{**})=x_k^*(u_k)\geq\epsilon$ for every $k$.

Let $N \in \N$. Define $x^* = \frac{1}{\sqrt{N}} \sum_{k=1}^N x_k^*$. Since the combined generators form a vfg family and the coefficients square-sum to $1$, $x^* \in G$. Evaluating $x^*$ on $y^{**}$ gives:
\[ x^*(y^{**}) = \frac{1}{\sqrt{N}} \sum_{k=1}^N x_k^*(y^{**}) \ge \frac{1}{\sqrt{N}} (N \epsilon) = \epsilon \sqrt{N} \]
However, since $y^{**} \in X_c^{**}$, we must have $x^*(y^{**}) \le \|y^{**}\|_{X_c^{**}} \le M_y$. This forces $\epsilon \sqrt{N} \le M_y$ for all $N$, which is absurd. 

Therefore, $S_m \to y^{**}$ in norm, which proves $y^{**} \in X_c$. Every $x^{**} \in X_c^{**}$ can be uniquely written as the sum of a vector in $X_c$ and a multiple of $z^{**}$. Thus, $\dim(X_c^{**} / X_c) = 1$.
\end{proof}
As part of the proof of the theorem we have the following which will be used in the next section.
\begin{corollary} \label{C3} For every $x^{**} \in X_c^{**}$ there exists a $c\in \R$ and a $y\in X_c $ such that 
\[ x^{**} = c 1_\N + y\]
\end{corollary}
\begin{remark}\label{R5} For $\s$ a branch we set $X_\s$ the subspace of $X_\alpha$ generated by the family $\{ e_k\}_
{k\in \s}$.It is not clear if the spaces $X_c, X_\s$ are isomorphic since they have different vfg families of $\alpha_c$ averages. However, it readily follows that $X_\s$ satisfies the same properties as $X_c$. In particular, the following holds.
\end{remark}
\begin{theorem}\label{T2}For every branch $\s \in \Gamma$ the space $X_\s $ is a quasireflexive complemented subspace of $X_\alpha$ with $\dim(X_\s^{**} / X_\s)  = 1$. Moreover for every $x^{**} \in X_\s^{**} \subset X_\alpha^{**}$ there are 
$c\in \R $ and $y\in X_\s$ such that
 \[ x^{**} = c 1_\s + y\]
 \end{theorem}

\section{Properties of $X_\alpha^*$}
In this section, we investigate the dual space $X_\alpha^*$. We show that its basis forms an $\ell_1$ spreading model and establish lower $\ell_2$ estimates for normalized block sequences. The main result of this section is the proof that $X_\alpha^*$ does not contain an isomorphic copy of $\ell_1$.

\subsection{$\ell_1$ spreading models in $X_\alpha^*$ } We start with the following, which is a consequence of the property that the basis of $X_\alpha$ is a $c_0$ spreading model.
\begin{proposition}\label{P_ell1_sm_dual}
The basis $(e_n^*)_n$ of $X_\alpha^*$ is an $\ell_1$ spreading model. Moreover, every normalized block sequence $(x_n^*)_n$ in $X_\alpha^*$ with $\inf_n\|x_n^*\|_\infty > \theta > 0$ is an $\ell_1$ spreading model.
\end{proposition}

\begin{proof}
The proof is a direct consequence of Remark \ref{R4} and the $c_0$ spreading model property of the basis $(e_n)_n$ of $X_\alpha$ (Proposition \ref{P4}).

For the basis $(e_n^*)_n$, we have $e_n^*(e_n) = 1 > 0$. Since $(e_n)_n$ is a normalized sequence which is a $c_0$ spreading model, Remark \ref{R4} directly implies that $(e_n^*)_n$ is an $\ell_1$ spreading model.

More generally, let $(x_n^*)_n$ be a normalized block sequence in $X_\alpha^*$ with $\inf_n\|x_n^*\|_\infty > \theta > 0$. We can choose a  sequence $(e_{k_n})_n$ in $X_\alpha$ such that $| x_n^*(e_{k_n})| > \theta$. Because the sequence $(e_{k_n})_n$ of $X_\alpha$ is a $c_0$ spreading model, the result follows from Remark \ref{R4}.

\end{proof}
\begin{remark} A consequence of the above is that there are many absolutely convex combinations of the basis of $X_\alpha^*$ with large norm. In fact,  the critical point in proving that $\ell_1 $ does not embed into $X_\alpha^*$ is to show that normalized block sequences of absolutely convex combinations are not equivalent to the $\ell_1$ basis.
\end{remark}

\subsection{Lower $\ell_2$ estimates}

\begin{lemma}
Let $(x_n^*)_n$ be a normalized block sequence in $X_\alpha^*$. Then $(x_n^*)_n$ admits a lower $\ell_2$ estimate with constant $1/4$.
\end{lemma}

\begin{proof}
We established in Proposition 3.3 that every normalized block sequence in $X_\alpha$ admits an upper $\ell_2$ estimate with constant $4$. By standard duality arguments for block basic sequences (see, e.g., \cite[Section 1.b]{LT77} or \cite[Section 3.2]{AK06}), this immediately yields a lower $\ell_2$ estimate with constant $1/4$ for normalized block sequences in the dual space $X_\alpha^*$.
\end{proof}

\begin{lemma}
Let $(x_n^*)_n \subset \operatorname{co}(G)$ be a block sequence in $X_\alpha^*$ such that:
\begin{enumerate}
    \item[(i)] $\|x_n^*\| > \theta > 0$ for all $n$,
    \item[(ii)] For every sequence of scalars $(\lambda_i)_{i=1}^m$ with $\sum_{i=1}^m \lambda_i^2 \le 1$, we have $\sum_{i=1}^m \lambda_i x_i^* \in \operatorname{co}(G)$.
\end{enumerate}
Then $(x_n^*)_n$ is equivalent to the $\ell_2$ basis.
\end{lemma}

\begin{proof}
We need to establish both an upper and a lower $\ell_2$ estimate for the sequence $(x_n^*)_n$.

\textbf{Upper $\ell_2$ estimate:}
Let $(a_i)_{i=1}^m$ be an arbitrary sequence of scalars, not all zero. Set $\lambda_i = \frac{a_i}{\left(\sum_{j=1}^m a_j^2\right)^{1/2}}$. Then $\sum_{i=1}^m \lambda_i^2 = 1$. By hypothesis (ii), the functional $y^* = \sum_{i=1}^m \lambda_i x_i^* \in \operatorname{co}(G)$. Since every functional in $G$ has norm at most 1, elements in its convex hull also have norm at most 1. Thus, $\|y^*\|_{X_\alpha^*} \le 1$.
This implies:
\begin{equation}
\left\| \sum_{i=1}^m a_i x_i^* \right\|_{X_\alpha^*} \le \left( \sum_{i=1}^m a_i^2 \right)^{1/2}
\end{equation}

\textbf{Lower $\ell_2$ estimate:}
The sequence $(x_n^*/\|x_n^*\|)_n$ is a normalized block sequence in $X_\alpha^*$. Applying Lemma 5.3 to this sequence and using hypothesis (i), we obtain
\begin{equation}
\left\|\sum_{i=1}^m a_i x_i^*\right\|_{X_\alpha^*}
\geq \frac14\left(\sum_{i=1}^m a_i^2\|x_i^*\|^2\right)^{1/2}
\geq \frac{\theta}{4}\left(\sum_{i=1}^m a_i^2\right)^{1/2}.
\end{equation}

The combination of the upper and lower bounds proves that $(x_n^*)_n$ is equivalent to the $\ell_2$ basis.
\end{proof}

\begin{lemma}\label{L_decomp}
Let $x^* \in X_\alpha^*$ with finite support and $\|x^*\| = 1$. Then for every $\epsilon > 0$, there exists a $y^* \in \operatorname{co}(G)$ satisfying the following:
\begin{enumerate}
    \item[(i)] $\|x^* - y^*\| < \epsilon$
    \item[(ii)] $y^* = z^* + w^*$, where $z^*$ and $w^*$ have the following properties:
    \begin{enumerate}
        \item[(1)] $z^*$ is an absolutely convex combination of the elements of the basis (i.e., $\|z^*\|_{\ell_1} \le 1$).
        \item[(2)] $w^* = \sum_{i \in F} \mu_i w_i^*$ with $\sum_{i \in F} |\mu_i| \le 1$, where $w_i^* = \sum_{j \in F_i} \lambda_{j,i} \alpha_{j,i}^* \in G$, and  $\min \supp(x^*) \le \min\{\min s(w_i^*), \min\operatorname{supp\,gen}(w_i^*)\}$.
    \end{enumerate}
\end{enumerate}
\end{lemma}

\begin{proof}
Since the norm of $X_\alpha$ is defined by $G$ and $G$ is symmetric, the unit ball of $X_\alpha^*$ is the weak*-closure of $\operatorname{co}(G)$. Choose a sequence of convex combinations of elements of $G$ converging weak* to $x^*$, and restrict each constituent functional to $\operatorname{ran}(x^*)$. Since $G$ is closed under interval restrictions, the resulting sequence remains in $\operatorname{co}(G)$ and still converges weak* to $x^*$. All its terms lie in a common finite-dimensional subspace, so the convergence is also in norm.
Thus, there exists $y_0^* \in \operatorname{co}(G)$ such that $\|x^* - y_0^*\| < \epsilon$. 

Let $I = \text{ran}(x^*)$. We define $y^* = y_0^*|_I$. By the definition of $G$, it is closed under restriction to intervals, which implies $y^* \in \operatorname{co}(G)$. Since $x^*$ is supported on $I$, we have $x^* = x^*|_I$, and therefore:
\begin{equation}
\|x^* - y^*\| = \|(x^* - y_0^*)|_I\| \le \|x^* - y_0^*\| < \epsilon
\end{equation}

Since $y^* \in \operatorname{co}(G)$ and $\supp(y^*) \subset \text{ran}(x^*)$, we can write $y^* = \sum_{i \in M} \mu_i v_i^*$ with $\mu_i \ge 0$, $\sum_{i \in M} \mu_i = 1$, and $v_i^* \in G$ restricted to $I$. 
Each $v_i^*$ is either a singleton $\pm e_n^*$ or generated by a successive vfg family of averages. Let $M_1$ be the indices where $v_i^*$ is a singleton, and $M_2$ be the indices where $v_i^*$ is generated by a vfg family. 

For $i \in M_2$, we write:
\begin{equation}
v_i^* = \sum_{j \in H_i} \lambda_{j,i} \alpha_{j,i}^*
\end{equation}
where $H_i = \{1, \dots, l_i\}$ and $\sum_{j \in H_i} \lambda_{j,i}^2 \le 1$. 
We separate the first generator from the rest:
\begin{equation}
v_i^* = \lambda_{1,i} \alpha_{1,i}^* + \sum_{j=2}^{l_i} \lambda_{j,i} \alpha_{j,i}^* = \lambda_{1,i} \alpha_{1,i}^* + w_i^*
\end{equation}
Let $F_i = \{2, \dots, l_i\}$. Because $\sum_{j \in F_i} \lambda_{j,i}^2 \le \sum_{j \in H_i} \lambda_{j,i}^2 \le 1$, and the remaining averages form a sub-family of a vfg family, we have $w_i^* \in G$. Let $F = \{i \in M_2 : w_i^* \neq 0\}$.

We now define:
\begin{equation}
z^* = \sum_{i \in M_1} \mu_i v_i^* + \sum_{i \in M_2} \mu_i \lambda_{1,i} \alpha_{1,i}^*
\end{equation}
\begin{equation}
w^* = \sum_{i \in F} \mu_i w_i^*
\end{equation}
Clearly $y^* = z^* + w^*$. Let us verify the properties.

\textbf{1) Property of $z^*$:} The singletons have $\ell_1$-norm 1. The first average $\alpha_{1,i}^*$ of any vfg family is an arithmetic mean of basis elements, so $\|\alpha_{1,i}^*\|_{\ell_1} \le 1$. Because $\sum \lambda_{j,i}^2 \le 1$, we have $|\lambda_{1,i}| \le 1$. The sum of the absolute weights is:
\begin{equation}
\sum_{i \in M_1} \mu_i + \sum_{i \in M_2} \mu_i |\lambda_{1,i}| \le \sum_{i \in M_1} \mu_i + \sum_{i \in M_2} \mu_i = 1
\end{equation}
Thus, $z^*$ is an absolutely convex combination of the elements of the basis.

\textbf{2) Property of $w^*$:} We have $w^* = \sum_{i \in F} \mu_i w_i^*$ with $\sum_{i \in F} \mu_i \le 1$. For $i \in F$, the functional $w_i^*$ begins with the second generator $j=2$. By the vfg property:
\begin{equation}
s(\beta_{2,i}^*) > 2^{\max \supp(\beta_{1,i}^*)}
\end{equation}
Because $v_i^*$ is supported inside $\text{ran}(x^*)$ and $\alpha_{1,i}^* \neq 0$, the first generator must intersect $\text{ran}(x^*)$. This means:
\begin{equation}
\max \supp(\beta_{1,i}^*) \ge \min \supp(\alpha_{1,i}^*) \ge \min \text{ran}(x^*) = \min \supp(x^*)
\end{equation}
Therefore, for any generator $j \ge 2$:
\begin{equation}
s(\beta_{j,i}^*) > 2^{\min \supp(x^*)} > \min \supp(x^*)
\end{equation}
Setting $\min s(w_i^*) = \min_{j \in F_i} s(\beta_{j,i}^*)$, we get $\min s(w_i^*) \ge \min \supp(x^*)$.
Successiveness also gives
\[
\min\operatorname{supp\,gen}(w_i^*)=\min\supp(\beta_{2,i}^*)>\max\supp(\beta_{1,i}^*)\geq\min\supp(x^*).
\]

\end{proof}

\begin{lemma}\label{L_extract_upper_l2}
Let $(x_n^*)_n$ be a normalized block sequence in $X_\alpha^*$. Suppose there exist block sequences $(z_n^*)_n$ and $(w_n^*)_n$ such that $\|x_n^* - (z_n^* + w_n^*)\| < \epsilon_n$, with $\sum_{n=1}^\infty \epsilon_n = \epsilon < 1/2$. Assume further that for every $n \in \N$, the pair $z_n^*, w_n^*$ satisfies the properties of the previous decomposition lemma: $z_n^*$ is an absolutely convex combination of basis elements, and $w_n^* = \sum_{i \in F_n} \mu_{i,n} w_{i,n}^*$ with $\sum_{i \in F_n} |\mu_{i,n}| \le 1$, $w_{i,n}^* \in G$, and $\min \supp(x_n^*) \leq \min\{\min s(w_{i,n}^*),\min\operatorname{supp\,gen}(w_{i,n}^*) \}$. Then there exists a subsequence $(w_{n_k}^*)_k$ which admits an upper $\ell_2$ estimate.
\end{lemma}

\begin{proof}
By absorbing the signs into the functionals $w_{i,n}^* \in G$ (since $G$ is symmetric) and introducing some potentially cancelling terms, we may assume without loss of generality that $\mu_{i,n} \ge 0$ and $\sum_{i \in F_n} \mu_{i,n} = 1$.

For each $n \in \N$, let $h_n$ be the maximum integer appearing in the support of any generator of any $w_{i,n}^*$ for $i \in F_n$.
Since $(x_n^*)_n$ is a block sequence, $\min \supp(x_n^*) \to \infty$ as $n \to \infty$. We can therefore inductively extract a subsequence $(n_k)_{k=1}^\infty$ such that for all $k > 1$:
\begin{equation}
\min \supp(x_{n_k}^*) > 2^{h_{n_{k-1}}}
\end{equation}

We claim that for any finite sequence of scalars $(\lambda_k)_{k=1}^m$ with $\sum_{k=1}^m \lambda_k^2 \le 1$, the linear combination $\sum_{k=1}^m \lambda_k w_{n_k}^*$ belongs to $\operatorname{co}(G)$.

Expanding the combination yields:
\begin{equation}
\sum_{k=1}^m \lambda_k w_{n_k}^* = \sum_{k=1}^m \lambda_k \left( \sum_{i \in F_{n_k}} \mu_{i, k} w_{i, k}^* \right)
\end{equation}
We can distribute this sum over the Cartesian product of the index sets $\Pi = F_{n_1} \times F_{n_2} \times \dots \times F_{n_m}$. For each multi-index $\vec{i} = (i_1, i_2, \dots, i_m) \in \Pi$, define the functional:
\begin{equation}
V_{\vec{i}}^* = \sum_{k=1}^m \lambda_k w_{i_k, k}^*
\end{equation}
By our subsequence selection, for each $k > 1$, the size of the initial generator of $w_{i_k, k}^*$ satisfies:
\begin{equation}
\min s(w_{i_k, k}^*) \ge \min \supp(x_{n_k}^*) > 2^{h_{n_{k-1}}} \ge 2^{\max \supp(w_{i_{k-1}, k-1}^*)}
\end{equation}
The support condition also gives, for $k>1$,
\[
\min\operatorname{supp\,gen}(w_{i_k,k}^*)\geq\min\supp(x_{n_k}^*)>2^{h_{n_{k-1}}}>h_{n_{k-1}},
\]
so the generators from successive selected terms are successive.
This rapid growth condition guarantees that the combined generators of $w_{i_1, 1}^*, \dots, w_{i_m, m}^*$ form a strictly successive very fast growing (vfg) family of averages. Since $\sum_{k=1}^m \lambda_k^2 \le 1$, the definition of the norming set yields $V_{\vec{i}}^* \in G$ for every $\vec{i} \in \Pi$.

We can now rewrite our linear combination as:
\begin{equation}
\sum_{k=1}^m \lambda_k w_{n_k}^* = \sum_{\vec{i} \in \Pi} \left( \prod_{k=1}^m \mu_{i_k, k} \right) V_{\vec{i}}^*.
\end{equation}
Let $P_{\vec{i}} = \prod_{k=1}^m \mu_{i_k, k} \ge 0$. The sum of these weights is exactly:
\begin{equation}
\sum_{\vec{i} \in \Pi} P_{\vec{i}} = \prod_{k=1}^m \left( \sum_{i \in F_{n_k}} \mu_{i, k} \right) \le 1.
\end{equation}
Thus, $\sum_{k=1}^m \lambda_k w_{n_k}^*$ is a convex combination of elements in $G$ (completed by the zero functional if the sum is strictly less than 1), which means it belongs to $\operatorname{co}(G)$.

Since every functional in $\operatorname{co}(G)$ has norm $\le 1$ in $X_\alpha^*$, we conclude:
\begin{equation}
\left\| \sum_{k=1}^m \lambda_k w_{n_k}^* \right\|_{X_\alpha^*} \le 1
\end{equation}
For an arbitrary finite sequence of scalars $(b_k)_{k=1}^m$, choosing $\lambda_k = b_k / (\sum b_j^2)^{1/2}$ yields the upper $\ell_2$ estimate with constant 1:
\begin{equation}
\left\| \sum_{k=1}^m b_k w_{n_k}^* \right\|_{X_\alpha^*} \le \left( \sum_{k=1}^m b_k^2 \right)^{1/2}
\end{equation}
This completes the proof.
\end{proof}

\subsection{ Non-embedding of  $\ell_1$  into $X_\alpha^*$}
We first show that for a bounded block sequence, convergence to zero on all branches yields that the sequence is weakly null. We then use this to derive that $\ell_1$ does not embed in $X_\alpha^*$.

\begin{lemma}\label{L12}
Let $(x_n^*)_n$ be a successive block sequence in $X_\alpha^*$ such that each $x_n^*$ is an absolute convex combination of the basis (i.e., $\sum |x_n^*(k)| \leq 1$) and, for all branches $\sigma$, $\lim_n1_\sigma(x_n^*) = 0$. Then $(x_n^*)_n$ is weakly null.
\end{lemma}

\begin{proof}
If $(x_n^*)_n$ is not weakly null, after passing to a subsequence $(y_k^*)_k$, there exists a functional $g \in X_\alpha^{**}$ (which we may assume is the weak*-limit of its finite projections) such that $|g(y_k^*)| \geq 1$ for all $k$. 
Fix $0 < \epsilon < 1$, and define the heavy node set $F_\epsilon = \{i \in \N : |g(e_i^*)| \ge \epsilon\}$. 
The evaluation of $g$ outside $F_\epsilon$ is strictly bounded:
\begin{equation}
|g|_{F_\epsilon^c}(y_k^*)| \le \sum_{i \notin F_\epsilon} |y_k^*(i)| \cdot |g(e_i^*)| \le \epsilon \|y_k^*\|_{\ell_1} \leq \epsilon.
\end{equation}

Since $\|g\| < \infty$ in the bidual of a tree space, the set $F_\epsilon$ cannot contain arbitrarily large antichains. (Otherwise we can define a vfg sequence of $\alpha_{inc}$ averages with their support in $F_\epsilon$ and blowing up the norm of $g$). Thus, from Dilworth's Lemma \cite{Dilworth50},  $F_\epsilon$ can be covered by a finite number of branches $\tau_1, \dots, \tau_d$ (up to a finite set of nodes that the supports of $y_k^*$ eventually bypass). 
Adding finitely many branches if necessary, we may enlarge $F_\epsilon$ to $\bigcup_{j=1}^d\tau_j$; the preceding estimate on its complement remains valid. For each $j$, the branch projection $P_{\tau_j}:X_\alpha\to X_{\tau_j}$ is bounded by Lemma~\ref{lem_branch_projection}, so $P_{\tau_j}^{**}g\in X_{\tau_j}^{**}$ and has the same coordinates as $g$ on $\tau_j$. By Remark~\ref{R5} and Theorem~\ref{T2}, it has the form $\lambda_j1_{\tau_j}+u_j$ with $u_j\in X_{\tau_j}$, whose coordinates tend to zero along $\tau_j$. Since distinct branches overlap only in finite initial segments, absorbing the finitely supported overlap corrections into these terms gives
\begin{equation}
g|_{F_\epsilon} = \sum_{j=1}^d \lambda_j 1_{\tau_j} + \sum_{j=1}^d c_j
\end{equation}
where $c_j\in c_0(\tau_j)$, extended by zero outside $\tau_j$.

Evaluating $g$ on $y_k^*$ yields:
\begin{equation}
g(y_k^*) = g|_{F_\epsilon^c}(y_k^*) + \sum_{j=1}^d \lambda_j 1_{\tau_j}(y_k^*) + \sum_{j=1}^d c_j(y_k^*)
\end{equation}
As $k \to \infty$:
\begin{enumerate}
    \item The tails $c_j(y_k^*) \to 0$ because the supports of $y_k^*$ tend to infinity and $y_k^*$ is an absolute convex combination
    \item By assumption, $1_{\tau_j}(y_k^*) \to 0$ for each branch $\tau_j$.
    \item $|g|_{F_\epsilon^c}(y_k^*)| \le \epsilon$.
\end{enumerate}

Taking the limit supremum of both sides gives:
\begin{equation}
\limsup_{k \to \infty} |g(y_k^*)| \le \epsilon
\end{equation}
But $1\leq |g(y_k^*)|$ for all $k$, which implies $1 \le \epsilon$. Since we fixed $\epsilon < 1$, this is a contradiction. Therefore, the original sequence $(x_n^*)_n$ is weakly null.
\end{proof}

\begin{proposition}
\label{on block sigmas determine weak nullness}
Let $(x_n^*)_n$ be a bounded block sequence in $X_\alpha^*$ such that for every branch $\sigma \in \Gamma$, $\lim_{n \to \infty} 1_\sigma(x_n^*) = 0$. Then $(x_n^*)_n$ is a weakly null sequence.
\end{proposition}

\begin{proof}
To show that the sequence $(x_n^*)_n$ is weakly null, it is sufficient to show that every subsequence admits a further weakly null subsequence. Fix an arbitrary subsequence $(x_{n_m}^*)_m$. We may assume that it is seminormalized, since otherwise it has a subsequence converging to zero in norm. Normalizing and retaining the same notation, we may further assume that $(x_{n_m}^*)_m$ is normalized; this reduction preserves both the branch-limit hypothesis and the desired weak-null conclusion.

Choose positive numbers $(\epsilon_m)_m$ with $\sum_{m=1}^\infty\epsilon_m<1/2$. By Lemma~\ref{L_decomp}, for each $m \in \mathbb{N}$, we can find functionals $z_m^*$ and $w_m^*$ such that $\|x_{n_m}^* - (z_m^* + w_m^*)\| < \epsilon_m$, where $z_m^*$ is an absolutely convex combination of basis elements, and the sequence of pairs $(z_m^*, w_m^*)_m$ satisfies the generator-size condition in Lemma~\ref{L_decomp}(ii)(2).

By Lemma~\ref{L_extract_upper_l2}, there exists a further subsequence (which for notational simplicity we will index by $k$) such that $(w_k^*)_k$ admits an upper $\ell_2$ estimate. Thus, $(w_k^*)_k$ is weakly null. 

Because $(w_k^*)_k$ is weakly null, $1_\sigma(w_k^*) \to 0$ for every branch $\sigma \in \Gamma$.
By our initial hypothesis, $1_\sigma(x_{n_k}^*) \to 0$ for every branch. Since $\|x_{n_k}^* - (z_k^* + w_k^*)\| \to 0$, this implies that also:
\begin{equation}
\lim_{k \to \infty} 1_\sigma(z_k^*) = 0
\end{equation}

Now, from Lemma \ref{L12}, $(z_k^*)_k$ is weakly null. 
Since both $(z_k^*)_k$ and $(w_k^*)_k$ are weakly null, and the error $\|x_{n_k}^* - (z_k^* + w_k^*)\| \to 0$, the sequence $(x_{n_k}^*)_k$ is also weakly null. This completes the proof.
\end{proof}

We start with the following, easily proved, statement.

\begin{lemma}\label{L11} Let $(x_n^*)_n $ be a block sequence of absolutely  convex combinations of the basis in $X_\alpha^*$. Let further  $0 < \epsilon <  1$ and $A= \{\s_i \}_{i\in F} $ a set of branches such that, for $i\in F$, $ \lim_n  |1_{\s_i}(x_n^*) |=\alpha_i > \epsilon$
 Then $\# A< \frac{1}{\epsilon}$
 \end{lemma}
 The next is a variant of a result appeared implicitly in \cite{LS} and also  stated in \cite{AGM}.

\begin{lemma}\label{lem_countable_branch_limits}
Let $(x_n^*)_n$ be a bounded block sequence in $X_\alpha^*$. There exists a subsequence $(x_{n_k}^*)_k$ and an at most countable set of branches $A$ such that:
\begin{enumerate}
    \item For every branch $\sigma \in A$, $\lim_{k \to \infty} |1_\sigma(x_{n_k}^*)| = a_\sigma > 0$.
    \item For every branch $\sigma \notin A$, $\lim_{k \to \infty} 1_\sigma(x_{n_k}^*) = 0$.
\end{enumerate}
\end{lemma}

\begin{proof}
If the sequence converges to zero in norm, then the conclusion is trivial. By a standard scaling argument, we may therefore assume that it is normalized. Apply Lemmas~\ref{L_decomp} and~\ref{L_extract_upper_l2} , choosing positive approximation errors with sum less than $1/2$. After passing to a subsequence and relabelling, we obtain block sequences $(z_n^*)_n$ and $(w_n^*)_n$ such that
\[
\|x_n^*-(z_n^*+w_n^*)\|\longrightarrow 0,
\]
where each $z_n^*$ is an absolutely convex combination of the basis and $(w_n^*)_n$ admits an upper $\ell_2$ estimate. In particular, $(w_n^*)_n$ is weakly null. Thus, for every branch $\sigma$, $1_\sigma(w_n^*)\to 0$, and hence
\[
1_\sigma(x_n^*)-1_\sigma(z_n^*)\longrightarrow 0.
\]
It therefore suffices to obtain the required subsequence and set of branches for $(z_n^*)_n$.

We proceed by constructing nested subsequences $M_0 \supset M_1 \supset M_2 \supset \dots$ and a sequence of distinct branches $\sigma_1, \sigma_2, \dots$ inductively. Let $M_0 = \N$ and $A_0 = \emptyset$.

For step $k \ge 0$, assume $M_k$ and $A_k = \{\sigma_1, \dots, \sigma_k\}$ have been defined. We evaluate the maximum possible limit supremum on all remaining branches:
\begin{equation}
s_{k+1} = \sup_{\sigma \notin A_k} \left( \limsup_{n \in M_k} |1_\sigma(z_n^*) |\right)
\end{equation}

If $s_{k+1} = 0$, the process terminates; we set the final subsequence to $M_k$ and $A = A_k$. 
If $s_{k+1} > 0$, there exists a branch $\sigma_{k+1} \notin A_k$ such that:
\begin{equation}
\limsup_{n \in M_k} |1_{\sigma_{k+1}}(z_n^*) |> \frac{s_{k+1}}{2}
\end{equation}
By the definition of the limit supremum, we extract a further subsequence $M_{k+1} \subset M_k$ such that the limit along this subsequence exists:
\begin{equation}
\lim_{n \in M_{k+1}}  |1_{\sigma_{k+1}}(z_n^*)| = a_{k+1} > \frac{s_{k+1}}{2}
\end{equation}
We then set $A_{k+1} = A_k \cup \{\sigma_{k+1}\}$ and continue the induction.

If the process does not terminate, Lemma~\ref{L11} yields  that 
\[ a_k \to 0\]

We now define the diagonal subsequence $(z_{n_k}^*)_{k=1}^\infty$, where $n_k$ is the $k$-th element of $M_k$, and let $A = \{\sigma_1, \sigma_2, \dots\}$. 
For any $\sigma_j \in A$, the diagonal sequence is eventually contained in $M_j$, yielding $\lim_{k \to \infty} |1_{\sigma_j}(z_{n_k}^*)| = a_j > 0$. 
For any branch $\sigma \notin A$, its limit supremum along the diagonal sequence is bounded by its limit supremum along $M_k$ for any $k$. Therefore, $\limsup_{k \to \infty} |1_\sigma(z_{n_k}^*)|  \le s_{k+1}$. Since $s_{k+1} \to 0$, we conclude that $\lim_{k \to \infty} |1_\sigma(z_{n_k}^* )| = 0$.
\end{proof}

\begin{theorem}\label{thm_l1_dual}The space $X_\alpha^*$ does not contain an isomorphic copy of $\ell_1$.

\end{theorem}

\begin{proof}
Assume, for the sake of contradiction, that $\ell_1$ embeds into $X_\alpha^*$. By the classical Bessaga-Pe\l{}czy\'{n}ski selection principle, there exists a normalized block sequence $(x_n^*)_n$ in $X_\alpha^*$ that is equivalent to the standard $\ell_1$ basis.

By Lemma~\ref{lem_countable_branch_limits}, pass to a subsequence and let $A$ be the corresponding at most countable set of branches. Passing to a further subsequence by a diagonal argument, we may assume that $1_\sigma(x_n^*)$ converges for every $\sigma\in A$; for $\sigma\notin A$, it already tends to zero. Set
\[
y_k^*=\frac{x_{2k}^*-x_{2k-1}^*}{2}.
\]
Then $(y_k^*)_k$ is a bounded block sequence equivalent to the $\ell_1$ basis, and $1_\sigma(y_k^*)\to0$ for every branch $\sigma$. By Proposition~\ref{on block sigmas determine weak nullness}, $(y_k^*)_k$ is weakly null, contradicting its equivalence to the $\ell_1$ basis.
\end{proof}

As we have mentioned in the introduction the space $JT$ is complementably  $\ell_2 $ saturated. The following is open and we believe that it has a positive answer.
\begin{question} Does every subspace $Y$ of $X^*_\alpha$  contain isomorphically $\ell_2$  complemented in $X^*_\alpha$?
\end{question}

\section{Properties of $X_\alpha^{**}$}
The final section is dedicated to the bidual space $X_\alpha^{**}$. We characterize its structure in terms of weak* limits along branches, show that the quotient $X_\alpha^{**}/X_\alpha$ is isometric to $c_0(\Gamma)$, and prove that $X_\alpha^{**}$ is complementably $\ell_2$-saturated.

Let us denote by $\sigma^{**} \in X_\alpha^{**}$ the weak* limit of the partial sums of the basis along a branch $\sigma \in \Gamma$. This limit exists since the sums over initial segments are uniformly bounded in $X_\alpha$ by Proposition \ref{prop_segment_sums}. The action of $\sigma^{**}$ on $x^* \in X_\alpha^*$ is exactly evaluated by $1_\sigma(x^*)$.

\begin{proposition}
$X_\alpha^{**} = \overline{\operatorname{span}} \left[ X_\alpha \cup \{\sigma^{**} : \sigma \in \Gamma\} \right]$.
\end{proposition}

\begin{proof}
If not, by the Hahn-Banach theorem there exists a functional $g \in X_\alpha^{***}$ such that $\|g\| = 1$ and 
\begin{equation}
\overline{\operatorname{span}} \left[ X_\alpha \cup \{\sigma^{**} : \sigma \in \Gamma\} \right] \subset \operatorname{Ker}(g).
\end{equation}

Since $\ell_1$ does not embed in $X_\alpha^*$ (by Theorem \ref{thm_l1_dual}), by the Odell-Rosenthal theorem \cite{OR75} there exists a sequence $(x_n^*)_n \subset X_\alpha^*$ with norm equal to $1$ such that $(x_n^*)_n$ converges weak* to $g$. 

Since $X_\alpha \subset \operatorname{Ker}(g)$, we have $g(e_m) = 0$ for all basis elements, meaning $x_n^*(e_m) \to 0$ for all $m$. By standard perturbation arguments, we may assume that $(x_n^*)_n$ is a block sequence. 

Moreover, since $\sigma^{**} \in \operatorname{Ker}(g)$ for all branches $\sigma \in \Gamma$, we have $\sigma^{**}(x_n^*) \to 0$, which means:
\begin{equation}
1_\sigma(x_n^*) \to 0 \quad \text{for all } \sigma \in \Gamma.
\end{equation}

Proposition \ref{on block sigmas determine weak nullness} shows that $(x_n^*)_n$ is weakly null in $X_\alpha^*$. This implies that its weak* limit in $X_\alpha^{***}$ must be $0$, which yields a contradiction since $\|g\| = 1$. This completes the proof.
\end{proof}

\begin{proposition}
Let $Q: X_\alpha^{**} \to X_\alpha^{**}/ X_\alpha$ be the canonical quotient map. Then:
\begin{enumerate}
    \item[(i)] The space $X_\alpha^{**}/ X_\alpha$ is isometric to $c_0(\Gamma)$, where $\Gamma$ is the set of branches of the dyadic tree.
    \item[(ii)] The quotient map $Q$ is a strictly singular operator.
\end{enumerate}
\end{proposition}

\begin{proof}
(i) Since for any branch $\sigma \in \Gamma$, $\sigma^{**} \notin X_\alpha$, and from the previous proposition $X_\alpha^{**}$ is generated by $X_\alpha$ and these branch limits, we conclude that:
\begin{equation}
X_\alpha^{**}/ X_\alpha = \overline{\operatorname{span}}\{ Q(\sigma^{**} ) : \sigma \in \Gamma \}
\end{equation}
In the sequel, we will denote $\tilde{\sigma} = Q(\sigma^{**})$.

Because the basis of $X_\alpha$ is bimonotone, for a finite family $\{\sigma_j\}_{j=1}^m$ of distinct branches and scalars $b_1, \dots, b_m$, the quotient norm satisfies:
\begin{equation}
\left\| \sum_{j=1}^m b_j \tilde{\sigma}_j \right\|_{X_\alpha^{**}/X_\alpha} = \lim_{n \to \infty} \left\| \sum_{j=1}^m b_j \sigma_{>n,j}^{**} \right\|_{X_\alpha^{**}}
\end{equation}
Evaluating the functional $\sum_{j=1}^m b_j \sigma_{>n,j}^{**}$ against the basis elements $e_k^*$ for $k \in \sigma_{>n, j}$ clearly yields $\left\| \sum_{j=1}^m b_j \sigma_{>n,j}^{**} \right\|_{X_\alpha^{**}} \ge \max_{1 \le j \le m} |b_j|$.

We claim that for every $\epsilon > 0$ there exists an $n \in \N$ such that:
\begin{equation}
\left\| \sum_{j=1}^m b_j \sigma_{>n,j}^{**} \right\|_{X_\alpha^{**}} \le (1+\epsilon) \max_{1 \le j \le m} |b_j|
\end{equation}
Choose $n \in \N$ sufficiently large such that the branches $\sigma_{>n,j}$ are pairwise incomparable and $\frac{2m}{2^n} < \epsilon$. Let $x^* = \sum_{i=1}^l \lambda_i \alpha_i^* \in G$ be an arbitrary functional defined by a successive vfg family of generators $\{\beta_i^*\}_{i=1}^l$. 

We examine the evaluation of $x^*$ on the sum of tails $\sum_{j=1}^m b_j \sigma_{>n,j}^{**}$. 
For the first generator $\alpha_1^*$, its support intersects the union of tails $\cup_{j=1}^m \sigma_{>n, j}$. If $\beta_1^*$ is a comparable average, it is supported on a single chain and thus intersects at most one branch tail $\sigma_{>n, j}$, giving an evaluation bounded by $|b_j| \le \max_j |b_j|$. If $\beta_1^*$ is an incomparable average, it intersects each branch tail in at most one node, so its evaluation is bounded by a convex combination of $|b_j|$, which is again $\le \max_j |b_j|$. Thus,
\begin{equation}
\left| \alpha_1^* \left( \sum_{j=1}^m b_j \sigma_{>n,j}^{**} \right) \right| \le \max_{1 \le j \le m} |b_j|
\end{equation}

For the subsequent averages $\alpha_i^*$ ($i \ge 2$), and their size satisfies
 \[s(\beta_i^*) > 2^{\max \supp(\beta_{i-1}^*)} > 2^{n+i-2} .\]
 Those  which  are comparable, their evaluations are similarly handled and bounded by
 \[\frac{\max_j |b_j|} {s(\beta_i^*)} <\frac{ \max_j |b_j|}{2^{n+i-2}} .\] 
  The  incomparable average each one  intersects the union of $m$ distinct branch tails in at most $m$ points, the evaluation of each subsequent incomparable average is bounded by $\frac{m}{2^{n+i-2}} \max_j |b_j|$. 
Summing these geometric errors gives a total contribution strictly less than $\epsilon \max_j |b_j|$.

Since $\sum \lambda_i^2 \le 1$, the total evaluation is bounded by $(1+\epsilon) \max_{1 \le j \le m} |b_j|$. Taking the limit as $n \to \infty$ shows that the quotient norm is exactly the supremum norm. Thus, $X_\alpha^{**}/ X_\alpha$ is isometric to $c_0(\Gamma)$.

(ii) The operator $Q$ is strictly singular if its restriction to any infinite-dimensional closed subspace is never an isomorphism. If $Q$ were an isomorphism on some infinite-dimensional subspace $Y \subset X_\alpha^{**}$, then $Y$ would be isomorphic to a subspace of $X_\alpha^{**}/ X_\alpha \cong c_0(\Gamma)$. Since every infinite-dimensional closed subspace of $c_0(\Gamma)$ contains an isomorphic copy of $c_0$, this would imply that $c_0$ embeds into $X_\alpha^{**}$. 

However, by a classical theorem of Bessaga and Pe\l{}czy\'{n}ski \cite{BP58}, because $X_\alpha^*$ does not contain an isomorphic copy of $\ell_1$, $X^{**}_\alpha$ does not contain an isomorphic copy of $c_0$. This contradiction proves that $Q$ must be strictly singular.
\end{proof}

\begin{lemma}
Let $(x_n)_n$ and $(x_n^*)_n$ be bounded block sequences in $X_\alpha$ and $X_\alpha^*$ respectively. Assume that both are equivalent to the $\ell_2$ basis and for every $m,n$, $x_m^*(x_n) = \delta_{m,n}$. Then the projection $P : X_\alpha^{**} \to \overline{\operatorname{span}}\{x_n\}_{n=1}^\infty$ defined as $P(x^{**}) = \sum_{n=1}^\infty x^{**}(x_n^*) x_n$ is a bounded linear operator.
\end{lemma}

\begin{proof}
Assume that $(x_n)_n$ and $(x_n^*)_n$ satisfy upper $\ell_2$ bounds with constants $C$ and $D$, respectively. Fix $x^{**}\in X_\alpha^{**}$. We will show that, for arbitrary $N\in\mathbb{N}$,
\[\Big\|\sum_{n=1}^Nx^{**}(x^*_n)x_n\Big\|\leq CD\|x^{**}\|.\]
Because $(x_n)_n$ is boundedly complete, this is sufficient to conclude that the series $\sum_{n=1}^\infty x^{**}(x_n^*)x_n$ converges and $\|P\|\leq CD$. Fix $x^*\in B_{X_\alpha^*}$. Then,
\begin{align*}
x^*\Big(\sum_{n=1}^Nx^{**}(x^*_n)x_n\Big) &= x^{**}\Big(\sum_{n=1}^N x^*(x_n)x^*_n\Big)\leq D\Big(\sum_{n=1}^N\big|x^*(x_n)\big|^2\Big)^{1/2}\|x^{**}\|\\
&= Dx^*\Big(\sum_{n=1}^N\frac{x^*(x_n)}{(\sum_{m=1}^N|x^*(x_m)|^2)^{1/2}}x_n\Big)\|x^{**}\|\\
&\leq CD \|x^*\|\Big(\sum_{n=1}^N\frac{|x^*(x_n)|^2}{(\sum_{m=1}^N|x^*(x_m)|^2)}\Big)^{1/2}\|x^{**}\|\leq CD\|x^{**}\|.
\end{align*}
\end{proof}

\begin{theorem}
Let $Y$ be a closed infinite-dimensional subspace of $X_\alpha^{^{**}}$. Then $Y$ contains a further subspace $Z$ isomorphic to $\ell_2$ which is complemented in $X_\alpha^{^{**}}$.
\end{theorem}

\begin{proof}
Since the quotient map $Q : X_\alpha^{^{**}} \to X_\alpha^{^{**}}/X_\alpha$ is a strictly singular operator, by standard perturbation arguments and Theorem~\ref{thm_complemented_l2_Xalpha}, there exists a normalized sequence $(z_n)_n$ in $Y$ and a block sequence $(w_n)_n$ in $X_\alpha$ such that:
\begin{enumerate}
    \item[(i)] $\sum_{n=1}^\infty \|z_n - w_n\| < \epsilon$ (for a sufficiently small $\epsilon > 0$), and
    \item[(ii)] The sequence $(w_n)_n$ is equivalent to the $\ell_2$ basis, and its biorthogonal functionals $(w_n^*)_n$ (defined on $X_\alpha$) are also equivalent to the $\ell_2$ basis.
\end{enumerate}
By the preceding lemma, because $(w_n)_n$ and $(w_n^*)_n$ are bounded block sequences equivalent to the $\ell_2$ basis satisfying $w_m^*(w_n) = \delta_{m,n}$, the canonical projection onto $\overline{\operatorname{span}}\{w_n\}_{n=1}^\infty$ extends to a bounded linear projection on $X_\alpha^{^{**}}$. Since $(z_n)_n$ is a small-norm perturbation of $(w_n)_n$, standard perturbation theory implies that the subspace $Z = \overline{\operatorname{span}}\{z_n\}_{n=1}^\infty$ is isomorphic to $\ell_2$ and complemented in $X_\alpha^{^{**}}$.
\end{proof}

\end{document}